\documentclass{article} 
\usepackage{amsmath,amstext,amsfonts,amsthm,amscd,amssymb}
\usepackage[utf8]{inputenc}
\usepackage{times}
\usepackage{cite}
\usepackage{xcolor}
\usepackage{graphicx}
\usepackage{tikz,tikzscale}
\usepackage{soul}
\usepackage[final,notref,notcite]{showkeys}
\title{{A Contraction Property} of an Adaptive Divergence-Conforming Discontinuous Galerkin Method for the Stokes Problem}
\author{Natasha Sharma \and Guido Kanschat}
\date{}
\let\phi\varphi
\def\grad{\nabla}%
\def\curl{\nabla\!\times\!}%
\def\div{\nabla\!\cdot\!}%
\def\hessian{\nabla^2}
\def\mvl#1{\left\{\!\!\left\{#1\right\}\!\!\right\}}
\def\jmp#1{\left[\!\!\left[#1\right]\!\!\right]}

\def\asterisk{*}
\def\ak{a^{IP}_k}
\def\aj{a^{IP}_j}

\def\bff{\boldsymbol f}

\def\bn{\boldsymbol n}

\def\be{\boldsymbol e}

\def\bu{\boldsymbol u}
\def\bv{\boldsymbol v}
\def\bw{\boldsymbol w}
\def\bx{\boldsymbol x}
\def\bH{\boldsymbol H}

\def\bV{\boldsymbol V}
\def\bu{\boldsymbol u}
\def\bv{\boldsymbol v}

\def\H{\text{H}}
\def\Hcurl#1{H^{\ifx#10\else#1,\fi\text{curl}}}
\def\Hdiv#1{\boldsymbol{H}^{\ifx#10\else#1,\fi\text{div}}}
\def\Hdivo#1{\boldsymbol{H}_0^{\ifx#10\else#1,\fi\text{div}}}
\def\Hrot#1{H^{\ifx#10\else#1,\fi\text{rot}}}

\def\T{\mathbb T}                           
\def\F{\mathbb F}                           
\def\R{\mathbb R}                           
\def\Q{\mathbb Q}                           
\def\P{\mathbb P}                           
\def\d{\partial}
\def\n{\boldsymbol n}

\def\Rka{\mathbb{R}_{k\asterisk}}
\def\ok{\text{osc}_k}

\newcommand{\norm}[1]{\bigl\|#1\bigr\|}

\def\form(#1,#2){\left(#1,#2\right)}
\def\forme(#1,#2){\left<#1,#2\right>}
\def\Mk{\mathbb{M}_k}   
\def\M1k{\mathcal{M}_{k+1}}                             
\def\d{\partial}
\def\n{\boldsymbol n}

\newtheorem{assumption}{Assumption}
\newtheorem{lemma}{Lemma}
\newtheorem{theorem}{Theorem}
\newtheorem{proposition}[theorem]{Proposition}

\newtheorem{remark}{Remark}

\begin{document}
\maketitle

\begin{abstract}
  We prove the contraction property for two successive loops of the adaptive algorithm for the
  Stokes problem reducing the error of the velocity. The problem is
  discretized by a divergence-conforming discontinuous Galerkin method
  which separates pressure and velocity approximation due to its
  cochain property. This allows us to establish the
  quasi-orthogonality property which is crucial for the proof of the contraction. We also establish the quasi-optimal complexity of the adaptive algorithm in terms of the degrees of freedom.
\end{abstract}

\section{Introduction}
\label{sec:introduction}

Numerical methods for solving the Stokes problem involve approximating
velocity and pressure by finite element spaces satisfying the inf-sup
condition. The
convergence analysis of an adaptive method for the Stokes problem
relies heavily on a quasi-orthogonality property which is
difficult to establish for this indefinite problem. The difficulty
here is due to the coupling of the pressure and velocity and
the inability to separate them for the analysis. However, in light of
the recent results~\cite{KanschatSharma14}, we are able to write the
discrete weak formulation for the Stokes problem as a pair of equations one
involving the divergence-free velocity and the other just involving
pressure. This allows us to obtain quasi-orthogonality for the
velocity only and thus to prove adaptive convergence.

First results in this direction were presented by Morin, Bänsch, and Nochetto
in~\cite{BaMoNo02} where
an Uzawa algorithm was used as an outer loop and continuous elements
of arbitrary degree were employed. In 2006 and 2007,
Kondratyuk and Stevenson~\cite{Ko06,KoSt08},
respectively, presented convergence and optimality for the Stokes
problem with continuous velocity approximations. However, all of these results have in common, that they analyze a
coupled adaptive and Uzawa iteration. On the other hand, Uzawa
iteration is not the most efficient solver for the Stokes
problem.
 
Pioneering work in the convergence and optimality of adaptive finite
element method using nonconforming Crouzeix-Raviart
element of the lowest order was initialized in a 2007
technical report by Hu and Xu~\cite{HuXu07}. Another work was by
Mao and Becker in~\cite{BeMa11} which included an analysis for
rectangular meshes and relying on the Crouzeix–Raviart and Rannacher–Turek elements. In~\cite{HuXu13}, Hu and Xu presented an
improved analysis
{ independent of the linear solver. It is directly based on the computed discrete solution and the separation of velocity and pressure depends on the  use of the Scott-Vogelius finite elements
and of non-conforming piecewise linear finite elements.} Meanwhile independently, Carstensen and co-workers
too presented a convergence and optimality analysis
in~\cite{CaPeRa13}. In both publications, the lowest order
Crouzeix-Raviart element was employed on triangles. 
In~\cite{CaPeRa13}, the analysis was
independent of pressure thanks to a new discrete Helmholtz
decomposition of the deviatoric matrices.

In this article, we present {{a} contraction property for adaptive}
divergence-conforming discontinuous Galerkin methods using finite
element exterior calculus tools to decompose velocity and pressure. In
particular, we are considering Raviart-Thomas elements of arbitrary
order. To our knowledge, this is the first higher order adaptive
method for the Stokes problem separating pressure and velocity.

The main ingredient for this current article is our recent
result~\cite{KanschatSharma14} which enables us to obtain a globally
discrete divergence-free velocity under the assumption that the
original fluid is incompressible.  Our approach is as follows: using
the cochain {property} of the finite element spaces, we rewrite the weak
formulation of the Stokes problem as a pair of variational equations:
one involving velocity and the other involving the pressure. The
velocity formulation computes velocity for the divergence-free
component of the right hand side $\bff$. The pressure formulation is
essentially a Poisson problem for which convergence and optimality
has already been presented. Then, we continue the analysis in the
divergence-free subspace. As an important ingredient, we need a
continuous interpolant, which is divergence-free. To this end, we
develop a new interpolation operator based on the element by Austin et
al.~\cite{AustinManteuffelMcCormick04}.

The main goal of the paper is to present a convergence of the Stokes
Problem in the same spirit as~\cite{BonitoNochetto10} based on the a
posteriori error estimator derived in our recent
paper~\cite{KanschatSharma14}. Some natural modifications are
introduced to extend the analysis to the Stokes
problem{. This extension is possible, since the
  divergence-free velocity depends on the divergence-free component of
  the right hand side only}. This way, our analysis need not involve
pressure component at all.

The paper is organized as follows. In section 2, we introduce the
necessary notation. Section 3 presents the Hodge decomposition for the
velocity space and a review { of }the divergence-conforming interior penalty
method. The main result of section 3 is the continuous
Raviart-Thomas space of order $m$ where $m \ge 2$ with commuting
interpolation operators. The { adaptive interior penalty method is described} in section 4 followed by {establishing the main ingredients of the
contraction property} in section 5. Section 6 addresses the optimality of the adaptive method and finally, in section 7 we discuss { some} numerical results which illustrate the
convergence of the adaptive method.
\section{Notation}
\label{sec:Notation}
Let $\Omega  \subset \R^d${, $d=2,3$ be a domain which can be expressed as a union of axis parallel rectangles and bricks, respectively,} with boundary
{$\Gamma$}. {While our analysis applies to the three-dimensional case, we will restrict the presentation to two dimensions and only comment on extensions where necessary.}
 In the following, $H^s(\Omega)$ denotes the $L^2$-based Sobolev space of differentiation order $s\ge0$. We will denote vector
and tensor{-}valued Sobolev spaces by $\bH^s(\Omega)$. The $L^2$-inner product {{on any $S\subset {\Omega}$ is denoted by
\begin{gather}
  \label{eq:l2-product}
  \form(f,g)_S = \form(f,g)_{L^2(S)} := \int_S f\odot g\,d\bx,
\end{gather}
}} 
where the generic multiplication operator ``$\odot$'' denotes the
product, the dot product, or the double contraction for scalar,
vector, and tensor functions, respectively. {{In particular, for $L^2$ inner products over the whole domain $\Omega$ we drop the subscript ``$\Omega$" that is, $\form(f,g)=\form(f,g)_{\Omega}.$}}
 Inner products in other
spaces are denoted by an index. In particular, on the subspace ${H^1_0(\Omega)}$ of 
$H^1(\Omega)$ with homogeneous boundary conditions such that a Poincaré inequality holds, we use the inner product
\begin{gather*}
  \form(f,g)_{ \bH^1_0(\Omega)} = \form(\nabla f, \nabla g).
\end{gather*}
For a differentiable scalar function $p$, we use the standard differential operator notation for the  vector curl $\curl p= (-\partial_2  p ,  \partial_1  p)$, the gradient
$\nabla p$, the symmetric tensor of second derivatives $\hessian
p$, the divergence $\div \bv$, and the Laplacian $\Delta p =
\div\nabla p$. For vectors and tensors, we define
\begin{gather*}
  \nabla\bv =
  \begin{pmatrix}
    \d_1 v_1 & \d_2 v_1 \\ \d_1 v_2 & \d_2 v_2
  \end{pmatrix}
  \qquad\text{and}\qquad
  \div
  \begin{pmatrix}
    a_{11} & a_{12} \\ a_{21} & a_{22}
  \end{pmatrix}
  =
  \begin{pmatrix}
    \d_1 a_{11} + \d_2 a_{12} \\ \d_1 a_{21} + \d_2 a_{22}
  \end{pmatrix}.
\end{gather*}
{ Although these differential operators can be analogously defined for three dimensions, we have restricted their definition to two dimensions since our analysis presented is restricted to two dimensions. 
}
Finally, we let $\Hdivo0(\Omega)$ be the space of all vector
fields with $L^2$ divergence and homogeneous normal boundary conditions {and let $H^{\text {curl}}(\Omega)$ denote the space of potentials on $\Omega$ whose vector curls live in $ L^2(\Omega)^2$}.
\section{The Stokes Problem}
Suppose $\Omega \subset \R^d$ is { as described in section~\ref{sec:Notation},}
 then the Stokes problem reads
\begin{gather*}
\begin{split}
- \Delta \bu + \grad p &= \bff  \quad  \text{in } \Omega,\\
\div \bu &=0  \quad  \text{in } \Omega, \\
  \bu&={\boldsymbol{0}}  \quad  \text{on } \Gamma,
\end{split}
\end{gather*}
with  $\bu $ denoting the velocity, $p$ the pressure and the external force  $\bff \in L^2(\Omega)^d$ acting on the fluid. 
Upon introducing boundary conditions, we obtain the following solution
spaces
\begin{align*}
  \bV &= \bigl\{ \bv\in\bH^1(\Omega) \;\big|\;
\bv = {\boldsymbol{0}} \quad  \text {on} \quad \Gamma
  \bigr\} 
  \\ Q &= {L}_0^2(\Omega) = {L}^2(\Omega)\setminus \mathbb{R}.
\end{align*}
The space $\bV$ admits the Hodge decomposition,
\begin{gather*}
  \bV = \bV^0 \oplus \bV^\perp,
\end{gather*}
where the divergence-free subspace $\bV^0$ consists of the curls of
functions in the potential spaces $H^1(\Omega)$ and
$H^{\text{curl}}(\Omega)$ in two and three dimensions,
respectively. If the domain $\Omega$ is not simply connected, $\bV^0$
will contain a finite number of harmonic functions as well. The space
$\bV^\perp$ is isomorphic to $Q$ and its elements $\bv$ can be characterized
uniquely by the conditions
\begin{align*}
  \exists q\in Q\quad : \quad&
  \div \bv = q,\\
  \forall \bw\in \bV^0\quad : \quad&
  \form(\bv,\bw)_{\boldsymbol H^1_0(\Omega)} = 0.
\end{align*}

The weak formulation requires us to find a suitable pair $(\bu,p) \in \bV \times Q$ such that
\begin{gather}
  \label{eq:Stokes}
  \form(\nabla\bu, \nabla\bv) - \form(p,\div \bv) - \form(q,\div \bu)
  = \form(\bff,\bv) \quad  \forall \  (\bv,q) \in \bV \times Q \quad \text{ holds.}
\end{gather}
We eliminate the pressure by restricting the above variational problem to
the subspace $\bV^0$. 
Then,
the formulation~\eqref{eq:Stokes} requires us to find $\bu \in \bV^0$ satisfying :
\begin{gather}
  \label{eq:Stokes_V0}
  \form(\nabla\bu, \nabla\bv) 
  = \form(\bff,\bv) \quad  \forall \  \bv \in \bV^0.
\end{gather}

Vice versa, by inserting $\div \bu=0$ into the weak formulation and
testing with $\bv\in \bV^\perp$, we obtain the pressure equation
\begin{gather*}
  - \form(p,\div\bv) = \form(\bff,\bv) \quad  \forall \  \bv \in \bV^\perp.
\end{gather*}
Solvability of this equation is due to the isomorphy of $\bV^\perp$
and $Q$. 
These two equations allow us to determine the velocity $\bu$ and the
pressure $p$ independently, based on the Hodge decomposition of
$\bff$. Next we will establish an analogous result for the discrete
problem.

\subsection{Discrete Hodge decomposition}

Let $\{\T_k\}_{k>0}$ be a uniformly shape regular family of partitions of
$\Omega$ into rectangular cells with $k$ denoting the level of mesh
refinement and obtained from refining an initial partition $\T_0$.
 We further assume that this partition is aligned with the
Cartesian axes i.e, any $T \ \in \T_k,$ can be expressed as
$T= T_x \times T_y$ with $T_x=[x_0, \ x_1]$, $T_y= [y_0, \ y_1],$
where $x_0 < x_1$ and $y_0 < y_1$. We remark that this condition seems
very restrictive and not necessary for our
discretization. Nevertheless, we use a lifting into a continuous
cochain complex suggested in~\cite{AustinManteuffelMcCormick04}, which
requires rectangular corners of all quadrilaterals or hexahedra. Let
furthermore
\begin{gather*}
\Q_{m,n}(T)=\{p(x)q(y) : p \in \P_m(T_x), \  q \in \P_n(T_y)  \}  
\end{gather*}
where $\P_\mu(I)$ denote the space of polynomials defined on $I
\subset \R$ of degree at most $\mu$. Let
\begin{gather*}
  RT_m(T) = \left\{
      \begin{pmatrix}
        p(x,y)\\q(x,y)
      \end{pmatrix}
      \bigg|
      p\in \Q_{m+1,m}(T), q\in \Q_{m,m+1}(T)
      \right\}
\end{gather*}
denote the Raviart-Thomas~\cite{RaviartThomas77} space of order $m$
defined on $T \in \T_k$. { Extension to three
dimensional finite elements is {straightforward~\cite{Nedelec80}}.

Associated with {the partition $\T_k$ are the discrete spaces
$\bV_{k}$ and $ Q_{k}$} defined by
\begin{gather}
\label{eq:discrete_velo}
\begin{split}
\bV_{k}&:=\{\bw \in {\Hdivo0(\Omega)}: \bw|_T \in RT_{m}(T), \ T \in \T_{k} \},\\
Q_{k}&:=\{q \in L_0^2(\Omega): q|_T \in \Q_{m,m}(T),\ T \in \T_{k}\}.
\end{split}
\end{gather}
These spaces are equipped with the canonical projection operators
$I_{V_k}$ into the Raviart-Thomas space
(cf. e.~g.~\cite{BoffiBrezziFortin13}) and the $L^2$-projection
$I_{Q_k}$ such that the following diagram commutes:
\begin{gather*} 
  \begin{CD}
    {\Hdivo0(\Omega)} @>\div>> L^2_0(\Omega)
    \\
    @V{I_{V_k}}VV @VV{I_{Q_k}}V
    \\
    \bV_k @>\div>> Q_k.
  \end{CD}
\end{gather*}

As a consequence of this diagram, we obtain
\begin{align*}
  \bV_k^0 &=\bigl\{ \bv_k\in \bV_k \;\big|\; (\div \bv_k,q_k) = 0
            \;\forall q_k\in Q_k \bigr\}
  \\
  &= \bigl\{ \bv_k\in \bV_k \;\big|\; \div \bv_k = 0 \bigr\}.
\end{align*}
 Associated with $\T_k$, we let $\F_k$ and $\overline{\F}_k$ denote the set of interior faces and set of all the faces respectively.
We close this section, by introducing short hand notation for the space 
$\bV(k)=\bV_k + \bV$ and, for the integrals over
$\T_k$ and $\F_k$,
\begin{xalignat*}2
  \form(f,g)_{\T_k} &:= \sum_{T\in\T_k}\form(f,g)_T,
  &
  \forme(f,g)_{\F_k} &:= \sum_{F\in\F_k}\forme(f,g)_F
  = \sum_{F\in\F_k} \int_F f\odot g\,ds.
\end{xalignat*}
Additionally, we introduce the discrete space  
\[\Sigma_{p}:=\{ {\underline{\tau}} \in L^2(\Omega, \R^{2 \times 2})\ | \ {\underline{\tau}}|_T \in \Q_{{p},{p}}(T)^{ 2 \times 2}, \quad T \in \T_k \}\] where $\Q_{{p},{p}}(T)^{ 2 \times 2}$ denotes the space of $2 \times 2$  valued functions with entries being polynomials of degree at most ${p}$ and {this} ${p}$ is { chosen to be $m+1$, with $m$ denoting the order of the Raviart-Thomas space $\bV_k$ introduced in~\eqref{eq:discrete_velo}.}

\subsection{The $\bH^1$-conforming subspace}

Establishing the reliability and the quasi orthogonality rely on
decomposing the error of the divergence-free velocity into an
$\bH^1$-conforming and a non-conforming component. Therefore, we now
describe the continuous subspace of higher-order Raviart-Thomas
elements which was introduced
in~\cite{AustinManteuffelMcCormick04}. To this end, let $m = 2$ and
define
\begin{gather*}
  \begin{alignedat}3
    \bV_{k}^c &:=\{\bw \in \bV &:
    \bw_{|_T} &\in RT_2(T),&
    \forall \; T &\in \T_k\},\\
    Q_{k}^c &:=\{q \in Q \cap H^1(\Omega) &:
    q_{|_T} &\in \Q_{2,2}(T),&
    \forall \; T &\in \T_k\}.
  \end{alignedat}
\end{gather*}

In~\cite{AustinManteuffelMcCormick04}, { the continuity of the elements in }the velocity
space $\bV_k^c$ is obtained by mixed Hermite/Lagrange interpolation,
as in Figure~\ref{fig:nodes} on the left.
\begin{figure}[tp]
  \centering
  \includegraphics[width=.45\textwidth]{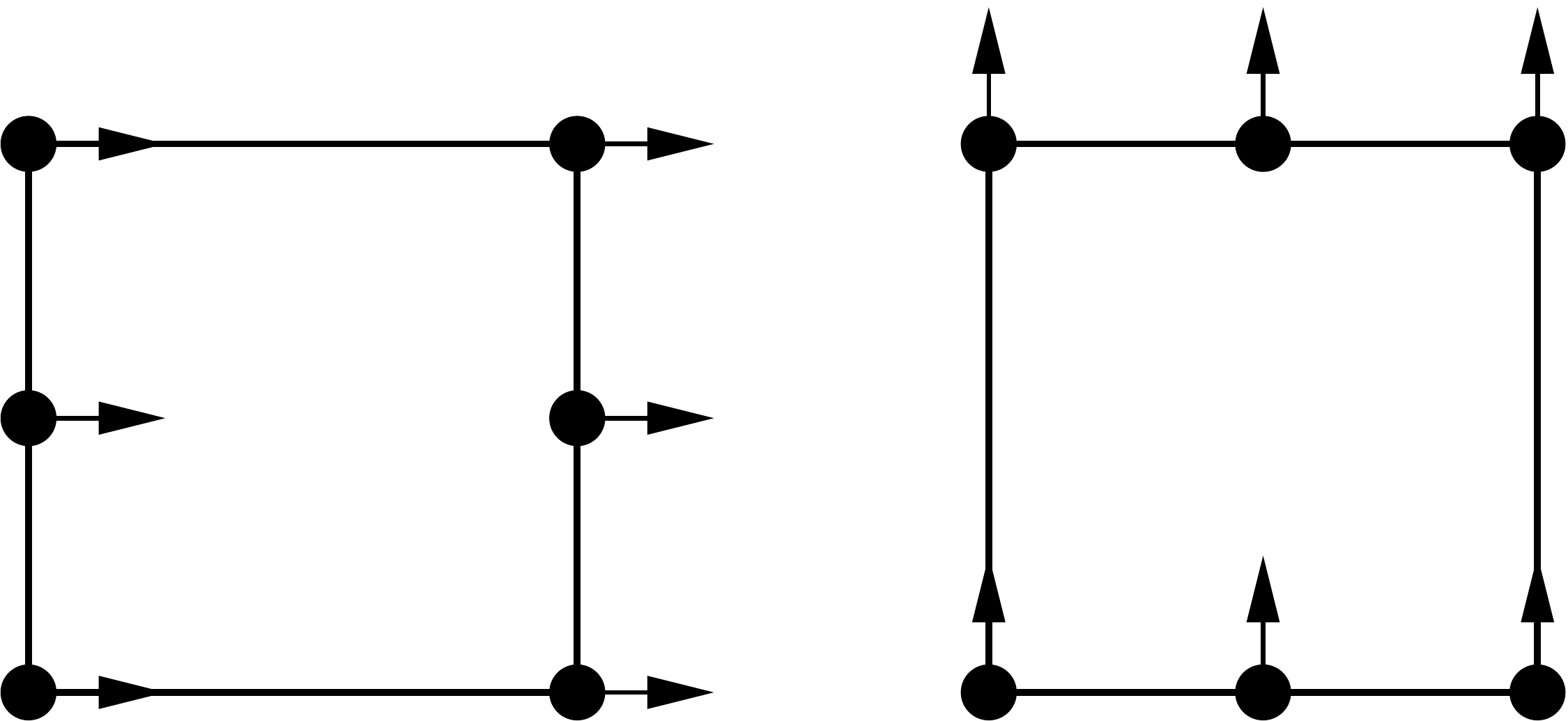}\hfill
  \includegraphics[width=.45\textwidth]{fig/nodes_crt}
  \caption[Degrees of freedom of the continuous Raviart-Thomas element]{The continuous Raviart-Thomas element of order two
    $\bV_k^c$ with degrees of freedom of Austin, Manteuffel, and
    McCormick (left) and of the commuting interpolation operator
    $\widehat I^c_{\bV}$ (right): symbols indicate point values of the
    velocity vector \tikz{\fill (0,0) circle [radius=.5ex];}, the
    normal component \tikz{\fill (0,1) ellipse (.4ex and .7ex);} and its
    normal derivative \tikz{\draw[thick,->] (1ex,0) -- (1ex,1.5ex);},
    the \textcolor{red}{values $\d_{\hat x} \hat u$ and $\d_{\hat y}\hat v$} \tikz{\draw[thick,->] (0,0) -- (0+.15,0+.15);
      \draw[thick,->] (0,0) -- (0+.15,0-.15); \draw[thick,->] (0,0) --
      (0-.15,0+.15); \draw[thick,->] (0,0) -- (0-.15,0-.15);}, as well as mean
    values of the normal component \tikz{\draw[fill=black,rounded
    corners=.02] (0,0) rectangle (1.5em,.3ex);} and its normal
  derivative \tikz{\fill (0,0) -- (1.5em,0) -- (.75em,.7ex) -- (0,0);} over edges.}
  \label{fig:nodes}
\end{figure}
Here, we need a pair of interpolation operators $I_{\bV}^c:\bV \to
\bV_k^c$ and $I_Q^c:Q\to Q_k^c$ such that the following diagram commutes:
\begin{gather*}
  \begin{CD}
    \bV @>\div>> Q
    \\
    @V{I_{\bV}^c}VV @VV{I_Q^c}V
    \\
    \bV_k^c @>\div>> Q_k^c.
  \end{CD}
\end{gather*}

We construct these interpolation operators in three steps: first, we
define it on smooth functions on the reference cell $\widehat T = [-1,1]^2$
choosing the node values on the right of Figure~\ref{fig:nodes} (see
equation~\eqref{eq:stokes_conv:1} below). These operators are extended
by push-forward on each cell to an interpolation operator on the whole
mesh. The last step consists of the
extension to less smooth functions.
Let $\hat x$ and $\hat y$ be
the coordinates on $\widehat T$ and
let $\hat u$ and $\hat v$ be the two components of the vector
field to be interpolated and let an index $k$ denote discrete
functions on $\widehat T$. Then, the interpolation operator
\begin{align*}  
  \widehat I_{\bV}^c : C^1(\widehat T) &\to RT_2(\widehat T) \\
  \begin{pmatrix}
    \hat u\\\hat v
  \end{pmatrix}
  &\mapsto
    \begin{pmatrix}
      \hat u_k\\\hat v_k      
    \end{pmatrix},
\end{align*}
is defined by the following 24 interpolation conditions:
\begin{gather}
  \arraycolsep2pt
  \label{eq:stokes_conv:1}
  \begin{split}
    \left.
      \begin{matrix}
        \hat u_k(\hat x, \hat y) & = & \hat u(\hat x, \hat y) \\
        \hat v_k(\hat x, \hat y) &=& \hat v(\hat x, \hat y) \\
        \partial_{\hat x}\hat u_k(\hat x, \hat y)
        &=& \partial_{\hat x}\hat u(\hat x, \hat y) \\
        \partial_{\hat y}\hat v_k(\hat x, \hat y)
        &=& \partial_{\hat y}\hat v(\hat x, \hat y)
      \end{matrix}
    \right\} &\quad \hat x=\pm1, \;\hat y=\pm1,\\
    \left.
      \begin{matrix}
        \displaystyle\int_{-1}^1
        \bigl[\hat u_k(\hat x, \hat y)-\hat u(\hat x, \hat y)\bigr]
        \,d\hat y & = & 0 \\
        \displaystyle\int_{-1}^1
        \bigl[\partial_{\hat x}\hat u_k(\hat x, \hat y)
        -\partial_{\hat x}\hat u(\hat x, \hat y)\bigr]
        \,d\hat y
        &=&0
      \end{matrix}
    \right\} &\quad \hat x=\pm1,\\
    \left.
      \begin{matrix}
        \displaystyle\int_{-1}^1
        \bigl[\hat v_k(\hat x, \hat y)-\hat v(\hat x, \hat y)\bigr]
        \,d\hat x & = & 0 \\
        \displaystyle\int_{-1}^1
        \bigl[\partial_{\hat y}\hat v_k(\hat x, \hat y)
        -\partial_{\hat y}\hat v(\hat x, \hat y)\bigr]
        \,d\hat x &=&0
      \end{matrix}
    \right\} &\quad \hat y=\pm1.
  \end{split}
\end{gather}
The finite element space as well as the interpolation operator have a
tensor product structure.
{
The velocity space restricted to any cell $T \in \T_k$ is
$\Q_{3,2}(T)\times \Q_{2,3}(T)$. The 12 node values involving the first
component of the velocity are
\begin{itemize}
\item the tensor product of the standard
Hermite interpolation in $x$-direction with an interpolation in
$y$-direction involving the function values at the end points and
\item  the
average over the interval $[-1,1]$.
\end{itemize}
}

These formulas determine polynomials of degrees 3 and 2, respectively,
in a unique way.  Therefore, this element is unisolvent. Furthermore,
it can be seen easily, that it is globally continuous with continuous
derivatives {of normal components}. Thus, its
divergence is continuous, and it is a biquadratic polynomial.

Choose now an interpolation operator $\widehat I_Q^c$ for $Q_{2,2}$ using the
following degrees of freedom: the values in the four vertices, the
mean values on the edges and the mean value on the whole reference
cell.
\begin{lemma}[Commutative Property]
  \label{lemma:commutation}
  The interpolation operators $\widehat I_{\bV}^c$ and $\widehat I_Q^c$
  admit the following commutative diagram:
  \begin{gather*}
    \begin{CD}
      C^1(\widehat T) @>\div>> C(\widehat T) \\
      @V{\widehat I_{\bV}^c}VV @VV{\widehat I_Q^c}V \\
      RT_2(\widehat T) @>\div>> Q_{2,2}(\widehat T).
    \end{CD}
  \end{gather*}
\end{lemma}
\begin{proof}
  Choose $(\hat u, \hat v)$ and let $\hat p = \div (\hat u, \hat
  v)^T$. Let furthermore $(\hat u_k, \hat v_k)^T = \widehat I_V^c(\hat u, \hat
  v)^T$ and $\hat p_k = \widehat I_Q^c(\hat p)$. Now, we use the interpolation
  conditions to show that $\hat p_k = \div (\hat u_k, \hat v_k)^T$.

  First, in all four vertices, there holds by the third and fourth
  condition in~\eqref{eq:stokes_conv:1}
  \begin{gather*}
    \hat p_k
    = \partial_{\hat x}\hat u + \partial_{\hat y} \hat v
    = \partial_{\hat x}\hat u_k + \partial_{\hat y} \hat v_k.
  \end{gather*}
  Then, for the mean value of $\hat p_k$ on the top edge of
  $\widehat T$, there holds by the first and last conditions
  in~\eqref{eq:stokes_conv:1}
  \begin{multline}
    \label{eq:1}
    \int_{-1}^1 \hat p_k(\hat x, 1) \,d\hat x
    = \int_{-1}^1 \hat p(\hat x, 1) \,d\hat x
    = \int_{-1}^1 \bigl[\partial_{\hat x}\hat u
    + \partial_{\hat y} \hat v\bigr] \,d\hat x
    \\
    = \hat u(1, 1) - \hat u(-1, 1)
    + \int_{-1}^1 \partial_{\hat y} \hat v(\hat x, 1) \,d\hat x
    \\
    = \hat u_k(1, 1) - \hat u_k(-1, 1)
    + \int_{-1}^1 \partial_{\hat y} \hat v_k(\hat x, 1) \,d\hat x.
  \end{multline}
  The same argument applies to the other three edges, such that there
  remains using Gauss' theorem and conditions five and seven
  in~\eqref{eq:stokes_conv:1} to deduce
  \begin{multline*}
    \int_{\widehat T}\hat p_k (\hat x, \hat y) \,d\hat x\,d\hat y
    = \int_{\widehat T}\hat p (\hat x, \hat y) \,d\hat x\,d\hat y
    = \int_{\partial \widehat T}
    \begin{pmatrix}
      \hat u\\\hat v
    \end{pmatrix}
    \cdot \n \,d\hat s
    \\
    = \int_{-1}^1 \bigl[\hat v(\hat x,1)-\hat v(\hat x,-1)\bigr]
    \,d\hat x
    + \int_{-1}^1 \bigl[\hat u(1, \hat y)-\hat v(-1,\hat y)\bigr]
    \,d\hat y
    \\
    = \int_{-1}^1 \bigl[\hat v_k(\hat x,1)-\hat v_k(\hat x,-1)\bigr]
    \,d\hat x
    + \int_{-1}^1 \bigl[\hat u_k(1, \hat y)-\hat u_k(-1,\hat y)\bigr]
    \,d\hat y
    \\
    = \int_{\widehat T} \div
    \begin{pmatrix}
      \hat u_k\\\hat v_k
    \end{pmatrix}
    \,d\hat x\,d\hat y.
  \end{multline*}
  Thus, we have proven that the defining node values for $\hat p_k$
  are those obtained from {$\div \widehat I_V^c (\hat u,\hat v)^T$}. Since
  these define $\hat p_k$ uniquely, we have indeed shown the statement
  of the lemma.
\end{proof}
So far, we have only studied interpolation of smooth functions on the
reference cell. Interpolation on the actual grid cell can be achieved
simply by pull-back of the interpolated function. Note that at this
point it is crucial, that the mesh cells are rectangular, since the
degrees of freedom at the corner points must be the normal derivatives
to the corresponding edges, not tangential, as Hermite interpolation
prescribes.

In order to extend the interpolation operator $\widehat I_V^c$ to a
continuous operator
\begin{gather}
  \label{eq:Clement}
  I_V^c : \bH^1(\Omega) \to \bV_k^c,  
\end{gather}
 { we summarize shortly the construction by
  J.~Schöberl in~\cite{Schoeberl08,Schoeberl10-multilevel-hcurl} and its adaptation
  to the case of continuous derivatives:
  like in the quasi-interpolation technique of Clément
  degrees of freedom consisting of function values in a vertex $x_i$ are
  replaced by weighted averages over a sufficiently small ball $B_i$ around
  this vertex:
  \begin{gather}
    \label{eq:vertex-integral}
    u(x_i) \longrightarrow \quad \int\limits_{B_i} \eta_i(\xi_i) u(\xi_i) \,d\xi_i,
    \qquad
    \partial_x u(x_i) \longrightarrow \quad
    \int\limits_{B_i} \eta_i(\xi_i) \partial_x u(\xi_i) \,d\xi_i.
  \end{gather}
  Now, consider the integration variables $\xi_i$ as a perturbation of
  the vertices $x_i$, respectively. 
  Since we require continuity of
  some derivatives of the velocity field, standard bilinear mappings
  of the original rectangles to the perturbed quadrilaterals is not
  sufficient. In fact, its normal derivative is discontinuous.   
  Instead, we define a new,
  curvilinear mesh by mapping each Cartesian cell $T$ with vertices
  $x_i$ to the cell $\widetilde T$ with vertices $\xi_i$. This mapping
  is realized by a
  bicubic function $\Psi_T$ defined by the Hermitian interpolation
  conditions
  \begin{xalignat*}{2}
    \Psi_T(x_i)&= \xi_i,
    & \partial_x\Psi_T(x_i) & = (1,0)^T, \\
    \partial_x\partial_y\Psi_T(x_i) &= 0,
    & \partial_y\Psi_T(x_i) & = (0,1)^T.
  \end{xalignat*}
  The function $\Psi$ defined on $\Omega$ by concatenation of the
  $\Psi_T$ is continuously differentiable by this definition.
  Clearly, $\operatorname{det}\nabla\Psi$ is uniformly positive if the
  balls for integration are sufficiently small. 
  Thus,
  $\nabla \Psi$ is invertible and its inverse is continuous over the
  whole mesh. Hence, the canonical interpolation operators for
  $\widetilde T$ defined by pull-back are defined consistently with
  its neighbors.
  The mollified interpolation
  operator is constructed by integrating the canonical integration
  over all such quadrilaterals generated by
  integration over the vertices as in~\eqref{eq:vertex-integral}, yielding the replacements
  \begin{align*}
    \int\limits_{F_{ij}} \bu\cdot\bn \,ds &\longrightarrow \quad
    \int\limits_{B_i} \eta_i
    \int\limits_{B_j} \eta_j
    \int\limits_{\widetilde F_{ij}}
    \bu\cdot\bn\,ds \,d\xi_i\,d\xi_j,\\
    \int\limits_{F_{ij}} \partial_n\bu\cdot\bn \,ds &\longrightarrow \quad
    \int\limits_{B_i} \eta_i
    \int\limits_{B_j} \eta_j
    \int\limits_{\widetilde F_{ij}}
    \partial_n\bu\cdot\bn\,ds \,d\xi_i\,d\xi_j.
  \end{align*}
  Here, $F_{ij}$ is the edge between $x_i$ and $x_j$ and
  $\tilde F_{ij}$ is the (curved) edge between $\xi_i$ and
  $\xi_j$. This set of degrees of freedom has the commutation property
  for each of the mapped quadrilaterals, such that it holds for the
  integrals by linearity. The interpolation operator $\tilde I^c_V$
  constructed this way is bounded on $H^1(\Omega)$, even on
  $W^{1,1}(\Omega)$, since all point evaluations have been replaced by
  integrals. It is not a projection though, since it does not preserve
  piecewise polynomials in $\bV^c_k$. This is achieved by applying its
  inverse on the discrete space. Thus,
  \begin{gather}
    \label{eq:interpolation-final}
    I^c_V = \left(\widetilde I^c_{V|_{\bV^c_k}}\right)^{-1} \widetilde I^c_V.
  \end{gather}
  Due to its tensor product structure, this construction applies to
  three dimensions as well.  }

{
\subsubsection{Hanging nodes}

Since we are using quadrilateral, even rectangular meshes, local grid
refinement inevitably leads to irregular meshes, i.e., not every edge
of a cell is also a complete edge of its neighboring cell. Consistent
with our implementation, we restrict this irregularity to one-irregular
meshes, that is, any edge of a cell is shared by at most two cells on
the other side of the edge. Thus, the generic situation is the patch
in Figure~\ref{fig:refined-edge}.
\begin{figure}[tp]
  \centering
  \includegraphics[width=.4\textwidth]{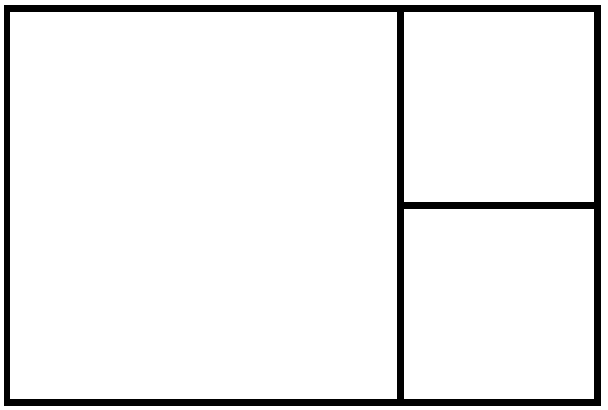}
  \caption{A one-irregular edge with a coarse cell on the left and two
    fine cells on the right.}
  \label{fig:refined-edge}
\end{figure}
We maintain the consistency of the finite element spaces by the
technique of eliminating hanging nodes. To this end, we observe that
continuity over an irregular edge can only be enforced if the traces
of the finite element spaces from both sides coincide. Therefore, only
linear combinations of basis functions are allowed on the refined side
which have traces in the shape function space on the coarse side. As a
result, the node functionals on the refinement edge are reduced to the
ones on the coarse side and the interpolation operator is only
evaluated on the coarse side as well. Hence, the computation of
equation~\eqref{eq:1} applies.  { The same technique can
  be applied to the construction of the operator $I^c_V$: here, no
  independent integral as in~\eqref{eq:vertex-integral} is introduced
  for a hanging vertex. Instead, the position of its perturbations is
  determined by its neighbors and conformity of $\Psi$ from both sides
  of the refinement edge.}

{The restriction to one-irregular meshes} simplifies the implementation
and the analysis. In particular,
shape regularity implies local quasi-uniformity for such meshes, such
that inverse estimates and the estimates obtained by
quasi-interpolation operators hold with the same asymptotics and only
modified constants.
}}

\subsubsection{Higher order elements}

Higher order versions of this element can be constructed by choosing
the velocity spaces $\Q_{m+1,m} \times \Q_{m,m+1}$ and the matching
pressure spaces $\Q_{m,m}$. Then, for every degree of freedom
corresponding to a mean value in~\eqref{eq:stokes_conv:1}, add moments
with respect to all Legendre polynomials of degree up to
$m-2$. Furthermore, add appropriate moments over the volume.
Unisolvence follows by the same tensor product argument and the
proof of Lemma~\ref{lemma:commutation} becomes an exercise in
excessive integration by parts.

\subsection{Divergence-conforming Interior Penalty Method}
{ For a given partition $\T_k$ of $\Omega$, we extend the notions of continuous and
differentiable function spaces to so called broken spaces, such that
for instance $C(\T_k)$ and $H^1(\T_k)$ are the spaces of functions
such that the restriction to each cell $T\in\T_k$ is in $C(T)$
and $H^1(T)$, respectively. For any two quantities $A$ and $B$, we use the notation $A\lesssim B$ if we can find a constant $C>0$ independent of the meshsize satisfying $A \le CB$.}
Let $F$ be an interior face in $\F_k$ such that the two cells $T_1$ and $T_2$ are adjacent
to $F$ in the point $x$. For a function $u\in {{C}(\T_k)}$, let $u_1(x)$
and $u_2(x)$ be the traces of $u$ in $x$ from cells $T_1$ and $T_2$,
respectively. Then, we define the sum operator
\begin{gather*}
  \mvl{u}(x) :=
  u_1(x)+u_2(x).
\end{gather*}
Let $\n_1$ and $\n_2$ be the outward normal vector
to $T_1$ and $T_2$, respectively. Then, by nature of its definition,
the sum operator applied to multiples of the normal vector
transforms to a jump:
\begin{gather*}
  \mvl{u\odot\n} = u_1\odot\n_1 + u_2\odot\n_2 = (u_1-u_2)\odot\n_1,
  \quad
  \mvl{\d_n u} = \d_{\boldsymbol{n_1}} (u_1-u_2),
\end{gather*}
where again ``$\odot$'' is the generic multiplication operator.
Whenever the jump appears squared and the ambiguity of the choice of
$T_1$ and $T_2$ is without effect, we use the short hand notation
\begin{gather*}
  \jmp u(x) = u_1(x)-u_2(x).
\end{gather*}

We also introduce the mesh sizes $h_T=|T|^{1/2}$ for $T \in \T_k$ and
$h_F$ associated with a face $F \in \F_k$ separating the two cells
$T^+$ and $T^-$ as the minimum of the length of $T^+$ and $T^-$
orthogonal to $F$. The {divergence-conforming interior
  penalty method henceforth referred to as $\Hdiv0$}-IP formulation of
the Stokes problem involves seeking $(\bu_k,p_k) \in \bV_k \times Q_k$
such that there holds
\begin{gather*}
  \hat{a}_k^{IP}(\bu_k,\bv) - (p_k,\div \bv)_{\T_k} - (q,\div \bu_k)_{\T_k} = \form(\bff,\bv)_{\T_k}
  \quad \forall\; (\bv,q) \in \bV_k\times Q_k,
\end{gather*}
where the elliptic bilinear form ${\hat{a}_k^{IP}(.,.)}$
implements the interior penalty method
\begin{multline*}
\hat{a}_k^{IP}(\bu,\bv) =  \form(\nabla \bu,\nabla \bv)_{\T_k} + 
\forme(\gamma h_F^{-1}\jmp{\bu},\jmp{\bv})_{\F_k} -
\\
\frac12\forme(\mvl{\nabla \bu},\mvl{\bv\otimes \n})_{\F_k} -
\frac12\forme(\mvl{\nabla \bv},\mvl{\bu\otimes \n})_{\F_k}.
\end{multline*}
The local mesh size $h_F$ is the extend of the cells adjacent to the
edge $F$ in orthogonal direction to $F$ and $\gamma>0$ is the penalty
parameter. In the presence of hanging nodes, we can simply choose
$h_F$ as the minimum of this extend over the two adjacent cells.
Since this mesh dependent bilinear form
${\hat{a}_k^{IP}(.,.)}$ is not well-defined on $\bV$, we
adopt the standard trick of introducing a lifting operator
$\mathcal L_S: \bV(k) \to \Sigma_{m+1}$ which is defined as,
\begin{gather*}
  \form(\mathcal L_S\bv,\tau)
  =\frac12\forme(\mvl{\tau},\mvl{\bv\otimes\n})_{\F_k}.
\end{gather*}
It is established
in ~\cite{PerugiaSchoetzau02hp,SchoetzauSchwabToselli03} that this
operator $\mathcal L_S $ admits the upper bound :
 \begin{gather}
 \label{eq:lift estimate}
 \norm{\mathcal L_S\bv }_{0,\Omega}^2 
 \le C_{l} 
  \norm{ h_F^{-\frac{1}{2}}\jmp{\bv} }^2_{\F_k} , \quad \bv \in \bV(k).
\end{gather}
where $C_l>0$ depends on the shape regularity of the mesh.

This operator enables us to extend the bilinear form  $ \hat{a}_k^{IP}$
to $\ak:\bV(k) \times \bV(k)\rightarrow \R$ defined as:
\begin{gather}
\label{eq:lifted-ip}
\ak(\bu,\bv) =  \form(\nabla \bu,\nabla \bv)_{\T_k} -{\form(\mathcal L_S\bv,\nabla \bu)} -{\form(\mathcal L_S\bu,\nabla \bv)} +
 \forme(\gamma h_F^{-1}\jmp{\bu},\jmp{\bv})_{\F_k}.
\end{gather}
The two forms are equivalent on the discrete space and only differ in
their smoothness assumptions on the continuous space.
Consequently, our IP method amounts to finding $(\bu_k,p_k) \in \bV_k \times Q_k$ :
 \begin{gather}
  \label{eq:lifted-ip-Stokes}
  \ak(\bu_k,\bv) - (p_k,\div \bv) - (q,\div \bu_k) ={ \form(\bff,\bv)}
    \quad \forall\; (\bv,q) \in \bV_k\times Q_k
 \end{gather}
 holds true.
 The discrete weak formulation \eqref{eq:lifted-ip-Stokes} is
 consistent with the continuous weak formulation \eqref{eq:Stokes} in
 the sense that the solution ${(}\bu, { p)} \in \bH_0^1(\Omega){\times Q}$ to
 \eqref{eq:Stokes} satisfies~\eqref{eq:lifted-ip-Stokes} albeit, only
 for $(\bv,q) \in \bH_0^1(\Omega)\times Q_k$
 (see~\cite[Sec.~2.2]{BonitoNochetto10}). This lack of consistency on
 $\bV_k$ motivates us to decompose the discretization error term into
 the $\bH^1$-conforming and non-conforming components and address each
 component separately.
  
 On the same space, we introduce the following discontinuous Galerkin
 (DG)-norm
\begin{gather*}
    \norm{\bu}_{1,k}^2 = \norm{\nabla\bu}_{\T_k}^2 + \gamma \norm{ h_F^{-\frac{1}{2}}\jmp{\bu}}^2_{\F_k}.
  \end{gather*}
It is shown in~\cite{BonitoNochetto10} that for sufficiently large penalty parameter $\gamma >0$, there exists a positive constant $C_{L} < \ 1$:
\begin{gather}
 \label{coer}
 \ak(\bu,\bu) \ge C_{L} \ \norm{\bu}^2_{1,k} \quad \bu \in \bV_k,
\end{gather}
also, for any $\gamma >1$, there exists a constant $C_{U} \ > \ 1,$
\begin{gather}
 \label{cont}
 \ak(\bu,\bv) \le C_{U} \ \norm{\bu}_{1,k} \norm{\bv}_{1,k}\quad \bu, \ \bv \in \bV(k){.}
\end{gather}
Below we quote the theorem which allows us to decouple the divergence
free velocity and pressure. The proof can be found
in~\cite{KanschatSharma14}.
\begin{proposition}[Decoupling of velocity and pressure]
  \label{Proposition}
  The velocity solution $\bu_k \in \bV_k$ and the pressure solution
  $p_k\in Q_k$ of the
  Stokes equation~\eqref{eq:lifted-ip-Stokes} can be computed
  independently. {In particular,} the velocity solves the reduced weak form
  \begin{gather}
  \begin{split}
    \label{eq:velocity}
    \ak(\bu_k,\bv_k^0) &= \form(\bff^0,\bv_k^0),
    \quad \forall  \bv_k^0\in \bV_k^0,
  \end{split}
  \end{gather}
  where $\bff^0$ is the divergence-free component of
  $\bff$.
\end{proposition}
Based on this result, we focus our {contraction property} on the
divergence-free velocity $\bu_k\in \bV_k^0$
satisfying~\eqref{eq:velocity}, such that for the discretization error
$\be_k=\bu -\bu_k$ there holds $\div\be_k = 0$.  The divergence-free
subspace $\bV^0_k \subset \bV_k$ is decomposed into
 \begin{gather}
 \begin{split}
 \label{decomp_space}
 \bV_k^0=\bV_k^c {\oplus} \bV_k^{\perp}
 \end{split}
 \end{gather} 
 where $\bV_k^{\perp}$ is the orthogonal complement of $\bV_k^c$ with respect to $\ak(.,.)$ in $\bV_k^0.$ { Consequently, we arrive at the following lemma.}

 \begin{lemma}[Decomposition of discrete velocity]
 \label{lemma:decomp_velocity1}
 { Any vector $\bv_k^{0} \in \bV_k^{0}$ can be uniquely decomposed in the form 
 \begin{gather}
 \label{eq:decomp}
 \bv_k^0 = \bv_k^c + \bv_k^{\perp}, 
 \end{gather}
where  $\bv_k^c \in 
  \bV_k^c$ is the $\ak(.,.)$ orthogonal projection of $\bv_k^0$ onto $\bV_k^c$ and
   $\bv_k^{\perp}$ satisfies the following inequality}
 \begin{gather}
   \label{eq:best_approx1}
   \norm{\bv_k^{\perp}}_{1,k}^2 \le  {C_{L}}^{-1}  \ak(\bv_k^0-\bw_k^c,\bv_k^0-\bw_k^c)  \quad \forall \bw_k^c \in \bV_k^{c}.
 \end{gather}
 \begin{proof}
   The unique representation~\eqref{eq:decomp} holds true due to
   decomposition~\eqref{decomp_space}. Additionally, thanks
   to~\eqref{coer} we have
  \begin{align}
  \norm{\bv_k^{\perp}}_{1,k}^2 &\le  {C_{L}}^{-1}  \ak(\bv_k^{0}-\bv_k^{c},\bv_k^{0}-\bv_{k}^{c})\nonumber\\
   &= {C_{L}}^{-1} \ \inf_{\bw_k^c \in \bV_k^{c}}  \ak(\bv_k^0-\bw_k^c,\bv_k^0-\bw_k^c),\label{decomp_a}
  \end{align}
  { where the best approximation property~\eqref{decomp_a} holds true because $\bv_k^{c}$ is the $\ak(.,.)$ orthogonal projection of $\bv_k$ on the space $\bV_k^c$.}
  \end{proof}
  \end{lemma}
 {
\section {Adaptive Interior Penalty Method} }
The implementation of the adaptive algorithm for the interior penalty (IP) method~\eqref{eq:velocity} is done according to the cycle :
\begin{gather}
\label{algo:afem}
  \text{SOLVE } \Longrightarrow \text{ ESTIMATE } \Longrightarrow \text{ MARK } \Longrightarrow \text{ REFINE. }
   \end{gather}
   { Based on an initial partition $\T_0$ of $\Omega$, we
     let the triple $\{(\T_k, \bV_k,\bu_k )\}_{k \ge 0}$ denote the
     sequence of partitions, discrete solution spaces and solutions
     respectively generated by a recursive application of adaptive
     algorithm~\eqref{algo:afem}.  } Here, the step `SOLVE' amounts to
   calculating the numerical solution of~\eqref{eq:velocity} which is
   realized by a direct solver. Below, we will discuss the remaining
   steps of the adaptive IP method including their properties which
   are crucial for proving the main results of this paper.

\subsection{Estimate}
{For the step `ESTIMATE' we consider the divergence-free, residual-type a posteriori error estimator as introduced in~\cite{KanschatSharma14}. We will discuss this estimator in this subsection, prove its reliability and other properties associated with it.

In order to introduce the divergence-free estimator, we let $\bu_k$ be the IP approximation to~\eqref{eq:velocity} and $\bff^0$ denote the divergence-free component of $\bff$.}
For { any} $T \in {\T}_{k}(\Omega),$ we denote the set of faces of the cell $T$ by $\F_k(T)$ and set the following notation:
\[
\eta_{k}^2(\bu_{k},T):=\eta_T^2(\bu_{k}) \  + \sum\limits_{F \in \F_k(T)} \eta_F^2(\bu_{k})
\]
\par\noindent
where the cell and edge residuals are: 
\begin{align*}
\eta_{T}(\bu_{k}) :=& \ h_T \ \| \bff^0 + \Delta \bu_{k}  \|_{0,T} , \\
 \eta_{F}(\bu_k) := & \ h_F^{1/2} \ \| {\mvl{ \partial_n \bu_{k}}} \|_{0,F} , \quad  F\in \F_k(T)
\end{align*}
respectively. Corresponding to any $\hat{\T}_k\subset \T_k$, we define the { divergence-free} estimator $\eta_k(\bu_k,\hat{\T}_k)$ as follows:
\begin{gather}
  \label{eq:St-estimateV0}
  \eta_{k}^2(\bu_k,\hat{\T}_k):= \sum_{T \in \hat{\T}_k} \eta_k^2(\bu_{k},T).
\end{gather}
In particular, we use the shorthand $\eta_k^2(\bu_k)\equiv \eta_k^2(\bu_k, \T_k)$.
{ For divergence-free right hand side $\bff$,
this estimator is optimal in the sense of achieving the optimal convergence rate as verified in~\cite{KanschatSharma14}.}
{ In this case, the estimator can be implemented as the standard elliptic estimator. For $\div \bff \neq 0$, the divergence-free part has to be extracted. On the other hand, it was verified in~\cite{KanschatSharma14} that this estimator is robust with respect to {nonzero and even large values of} $\div\bff$.
}

{ We also modify the quasi-interpolation operator
  $I_{\bV_k}^c:\bV(k) \rightarrow \bV_k^c$ in
  equation~\eqref{eq:interpolation-final}, such that it combines
  $\bH^1$-stability on
  $\bV$ with the treatment of the jumps
  following Lemma 6.6 in~\cite{BonitoNochetto10}.} Accordingly, there holds for some
  $C_{\text{interp}} >0$ depending only on the shape regularity of
  $\T_k$,
\begin{align}
\sum_{|\beta|} \norm{D^{\beta}(\bu_k- I_{\bV_k^c}\bu_k)}_{\T_k}^2 &\le C_{\text{interp}}\norm{ h_F^{\frac{1- |\beta|}{2}}\jmp {\bu_k}}^2_{\F_k}, \quad |\beta|\in\{0,1\}, \label{interp}\\
 \norm{(\bv_k- I_{\bV_k^c}\bv_k)}_{0,T} &\le 
 C_{\text{interp}} \norm{h D(\bv_k- I_{\bV_k^c}\bv_k)}_{L^2(\omega_T)} \quad \forall\bv_k \in\bV(k) \label{interp_conform},
\end{align}
where $\omega_T$ is defined as follows.
For any $T\in\T_k$,  
\begin{align}
&\omega_T \text{ denote the patch of }\text{cells that share an edge with }T \text{ and, } \label{defn:patch1}\\
& \text{for any subset of cells } \R_k \subset\T_k, \ \text{let }
\omega(\R_k)= \cup_{T \in \R_k }^{}\omega_T.\label{defn:patch2}.
\end{align}
\begin{remark}
In particular, taking $\bw_k^c=I_{\bV_k}^c \bv_k$ in~\eqref{eq:best_approx1}, using the coercivity~\eqref{coer} and taking into consideration the interpolation estimate~\eqref{interp}, we have
  \begin{gather}
  \label{eq:best_approx}
  \norm{\bv_k^{\perp}}_{1,k}^2 \lesssim  \norm{h_F^{-\frac{1}{2}}\jmp{\bv_k}}_{\F_k}^2.
  \end{gather}
\end{remark}
For the rest of the subsection, we present the properties of the estimator needed to prove the main two results of this article. We begin with the following lemma which bounds the jump terms from above by the estimator~\eqref{eq:St-estimateV0}. 
The proof follows the arguments given in~\cite{BonitoNochetto10}. We provide it below for completeness with some natural modifications.

\begin{lemma}[Estimator Control of the jump terms]
 \label{lemma:estimator_jump}
  Let $\bu_k$ be the IP approximation to the Stokes problem~\eqref{eq:velocity} and
  $ 0 < C_L <1$ be the coercivity constant from \eqref{coer}. Then,
  there is a constant $C_{J} >0$ depending only on the shape
  regularity of $\T_k$ such that for $\gamma > \frac{2C_J}{C_L}$,
  there holds
 \begin{gather}
 \label{eq:jumpcontrol}
 \gamma \norm{h_F^{-\frac{1}{2}}\jmp{\bu_k}}_{\F_k}^2  \le \frac{2C_J}{C_L}  \eta_k^2.
 \end{gather}
\end{lemma}
\begin{proof}Due to coercivity of $\ak(.,.)$ and $\jmp{I_{\bV_k}^c \bu_k}=0$
we have
{
\begin{align}
C_L \gamma \norm{h_F^{-\frac{1}{2}}\jmp{\bu_k}}_{\F_k}^2 &=  C_L\gamma \norm{h_F^{-\frac{1}{2}}\jmp{\bu_k - I_{\bV_k}^c \bu_k}}_{\F_k}^2 \nonumber \\
&\le C_L\norm{\bu_k - I_{\bV_k}^c \bu_k}_{1,k}^2 \nonumber \\
&\le \ak(\bu_k - I_{\bV_k}^c \bu_k,\bu_k - I_{\bV_k}^c \bu_k) \nonumber \\
&=  \Big((\bff^0, \bu_k - I_{\bV_k}^c \bu_k)_{\T_k} \  - \ 
\ak( I_{\bV_k}^c \bu_k,\bu_k - I_{\bV_k}^c \bu_k)\Big).\label{eq: ineq1}
\end{align}
Note that the first term on the right hand side of~\eqref{eq: ineq1}
is a consequence of~\eqref{eq:velocity} with the choice of the
divergence-free test function $\bv_k^0=\bu_k - I_{\bV_k}^c \bu_k $. It
also applies on faces with hanging nodes.
Now since $\mathcal L_S(I_{\bV_k}^c \bu_k)=0$ and $\jmp{I_{\bV_k}^c \bu_k}=0,$ the last term of~\eqref{eq: ineq1} can be written as 
\begin{align}
\ak(I_{\bV_k}^c \bu_k,\bu_k - I_{\bV_k}^c \bu_k)&=
\Big((\nabla I_{\bV_k}^c \bu_k, \nabla(\bu_k - I_{\bV_k}^c \bu_k) )_{\T_k} -({\mathcal{L}}_S \bu_k,\nabla I_{\bV_k}^c \bu_k)_{\T_k}\Big) \nonumber \\
&=\Big((\nabla \bu_k, \nabla(\bu_k - I_{\bV_k}^c \bu_k) )_{\T_k} -({\mathcal{L}}_S \bu_k,\nabla I_{\bV_k}^c \bu_k)_{\T_k}  \nonumber\\
&- \norm{\nabla(\bu_k - I_{\bV_k}^c \bu_k)}_{\T_k}^2\Big) \nonumber \\
&=
\Big( (-\Delta \bu_h, \bu_k - I_{\bV_k}^c \bu_k)_{\T_k} + ({\mathcal{L}}_S \bu_k,\nabla (\bu_k -I_{\bV_k}^c \bu_k))_{\T_k}  \nonumber \\
&- \norm{\nabla(\bu_k - I_{\bV_k}^c \bu_k)}_{\T_k}^2 
\Big) +  \frac12\forme(\mvl{\nabla \bu},\mvl{\bv\otimes \n})_{\F_k} \label{eq: iq2}
\end{align}
where the terms on the right hand side of~\eqref{eq: iq2} arise due to
a cellwise application of Green's formula. Hence, we have
\begin{align*}
&(\bff^0, \bu_k - I_{\bV_k}^c \bu_k)-\ak( I_{\bV_k}^c \bu_k,\bu_k - I_{\bV_k}^c \bu_k)\\
&\le \eta_k \big( \norm{h_T^{-1}(\bu_k - I_{\bV_k}^c \bu_k)}_{\T_k}^{}\ +\norm{ h_F^{-\frac{1}{2}}\jmp{\bu_k - I_{\bV_k}^c \bu_k}}_{\F_k}^{}\big) \ +\\
&\norm{\nabla(\bu_k - I_{\bV_k}^c \bu_k)}_{\T_k}^2  +
\norm{{\mathcal{L}}_S\bu_k}_{\T_k}\norm{\nabla(\bu_k - I_{\bV_k}^c \bu_k)}_{\T_k},
\end{align*}
where we have used the shortland $\eta_k\equiv \eta_k(\bu_k)$.
Finally, due to the properties satisfied by the interpolation operator and the stability estimate~\eqref{eq:lift estimate} for $L_S$, we can conclude the desired inequality.
}
\end{proof}
 { The proof of reliability of the estimator $\eta_k(\bu_k)$ follows the standard argument of decomposing the discretization error $\bu- \bu_k$ into a conforming and non-conforming component. Here, it is based on the space $\bV_k^c$ with commuting interpolation operator $I_{\bV_k }^c$.
 }
 \begin{proposition}[Reliability of the estimator] Let $\bu_k$ be the IP approximation to the Stokes
   problem and $\bu$ be the solution to the weak formulation of the
   Stokes problem. Then, there is a constant $C_{rel}>0$  depending only on the shape regularity of $\T_k$ such that
\begin{gather}
\label{eq:reliability}
\ak(\bu-\bu_k,\bu-\bu_k) \le C_{rel}\ \eta_k^2
\end{gather}
holds.
\end{proposition}
\begin{proof}
  We recall the notation for the discretization error $\be_k=\bu - \bu_k$ and we decompose
  $\bu_k=\bu_k^c+\bu_k^{\perp}$ into its $\bH^1$-conforming and
  non-conforming components, respectively, so that by the partial
  Galerkin orthogonality enjoyed on ${{\bH_0^1(\Omega)}} \cap \bV_k^0$, we
  have
\begin{align}
\ak(\be_k,\be_k)&=
\ak(\be_k,\be_k^{c}-\bu_k^{\perp} ) \nonumber\\
&=\ak(\be_k, \be_k^{c} - I_{\bV_k}^c\be_k^{c} )  - \ak(\be_k, \bu_k^{\perp}) \label{eq:rel1}.
\end{align}
Here, $\be_k^{c} = \bu - \bu_k^{c} \in \bH_0^1(\Omega)$ and $I^c_{\bV_k}$ is the interpolation operator introduced in~\eqref{interp}. We will now provide upper bounds for both these terms. For the first term of~\eqref{eq:rel1}, thanks to Cauchy-Schwarz inequality, Young's inequality and partial Galerkin
orthogonality, we have
\begin{align}
\ak(\be_k, \be_k^{c} - I_{\bV_k}^c\be_k^{c} ) &= (\bff^0 + \Delta \bu_k,\be_k^{c} - I_{\bV_k}^c\be_k^{c})_{\T_k} 
-({{\mvl{\partial_n \bu_k}}},\be_k^{c} - I_{\bV_k}^c\be_k^{c})_{\F_k} +\nonumber\\
&({{{\mathcal{L}}_S}}\bu_k, \nabla(\be_k^{c} - I_{\bV_k}^c\be_k^{c}))_{\T_k}\nonumber\\
&\lesssim \eta_k \Big(\norm{h_T^{-1}(\be_k^{c} - I_{\bV_k}^c\be_k^{c})}_{\T_k} + \norm{h_F^{-1/2}(\be_k^{c} - I_{\bV_k}^c\be_k^{c})}_{{\F_k}}\Big) +\nonumber\\
& \norm{h_F^{-\frac{1}{2}}\jmp{\bu_k}}_{F_k} 
\norm{\nabla(\be_k^{c} - I_{\bV_k}^c\be_k^{c})}_{{\T_k}}\nonumber\\
&\lesssim \Big(\eta_k +  \norm{h_F^{-\frac{1}{2}}\jmp{\bu_k}}_{\F_k}\Big)\ \norm{\nabla (\be_k + \bu_k^{\perp})}_{{\T_k}}\label{eq:rel2}
\end{align}
 where the upper bound~\eqref{eq:rel2} arises due to the interpolation approximation properties{~\eqref{interp}-\eqref{interp_conform} and the following scaled trace inequality
 \begin{align*}
\norm{ h_F^{-1/2}(\be_k^{c} - I_{\bV_k}^c\be_k^{c})}_{\F_k}&\lesssim  \norm{\nabla (\be_k^{c} - I_{\bV_k}^c\be_k^{c})} +
\norm{ h_T^{-1}(\be_k^{c} - I_{\bV_k}^c\be_k^{c})},
 \end{align*}
 which is based on piecewise polynomial properties and
   thus independent of the presence of hanging nodes. Nevertheless,
   the constant depends on local quasi-uniformity, such that the
   assumption of one-irregularity keeps the constant bounded. Now, thanks to the coercivity~\eqref{coer} and
    continuity~\eqref{cont} of $\ak(.,.)$ and Young's inequality to~\eqref{eq:rel2} we derive the following upper bound}
   \begin{align}
 \ak(\be_k, \be_k^{c} - I_{\bV_k}^c\be_k^{c} )  &\lesssim \eta_k^2 +  \norm{h_F^{-\frac{1}{2}}\jmp{\bu_k}}_{F_k}^2 + \frac{1}{4}\ak(\be_k, \be_k) \nonumber\\
   & \ + \frac{1}{4}\ak(\bu_k^{\perp},\bu_k^{\perp}) \label{eq:rel3}
   \end{align}
   
 Regarding the second term of~\eqref{eq:rel1}, we use the coercivity~\eqref{coer} and
 continuity~\eqref{cont} of $\ak(.,.)$ together with Young's
 inequality to obtain
\begin{align}
{-}\ak(\be_k, \bu_k^{\perp}) &{\lesssim
\norm{\be_k}_{1,k}\norm{\bu_k^{\perp}}_{1,k}} \nonumber
\\
 &\lesssim
\frac{1}{4}\ak(\be_k,\be_k) +  \ \norm{\bu_k^{\perp}}_{1,k}^2.\label{rel2}
\end{align}
{
Now, gathering the estimates~\eqref{eq:rel3} and~\eqref{rel2} and plugging it in~\eqref{eq:rel1} it follows that
\begin{align}
\ak(\be_k,\be_k) \lesssim \eta_k^2
+\norm{\bu_k^{\perp}}_{1,k}^2  +  \norm{h_F^{-\frac{1}{2}}\jmp{\bu_k}}_{\F_k}^2 \ \text{holds.}
\label{eq:rel4}
\end{align}
Finally, the proof can be concluded by using the upper bound~\eqref{eq:best_approx} for bounding the second term of~\eqref{eq:rel4} in conjunction with using the estimator control of the jump terms~\eqref{eq:jumpcontrol}.} 
\end{proof}
Given a discrete solution $\bu_k\in \bV_k$ corresponding to the partition $\T_k$ of $\Omega$,
the following lemma ensures the localization of the upper bound for the discrete error $\bu_*-\bu_k$ in the DG norm, where $\bu_* \in \bV_*$ solves~\eqref{eq:Stokes_V0} with respect to $\T_*$, which is a refinement of $\T_k$. 
 This lemma is needed to prove the optimality of the adaptive IP method in section 5 and follows the same approach described in~\cite{BonitoNochetto10}. We present it below for completeness.
\begin{lemma}[Quasi-Localized Upper Bound]
Let $\T_{\asterisk}$ and $\T_k$ be partitions of  $\Omega$ such that $\T_{\asterisk}$ is obtained by refining $\T_k$. Also, let $\Rka$ be the set of refined cells needed to obtain $\T_{\asterisk}$ from $\T_k.$ Let $\bu_k \in \bV_k$ and $\bu_{\asterisk} \in \bV_{\asterisk}$  be the IP approximations to~\eqref{eq:velocity} with respect to the partitions $\T_k$ and $\T_{\asterisk}$ respectively. Then, we can find $C_{ub}>0$ depending only on $\Omega$ such that 
\begin{gather}
\label{eq:local_ub}
\norm{\bu_{\asterisk}^c-\bu_{k}  }^2_{1,k} \le C_{ub} \ \big(\eta^2_{k}(\bu_k, \Rka)  + \gamma^{-1} \eta^2_k\big),
\end{gather}
where $\bu_{\asterisk}^c= \bu_{\asterisk} -\bu_{\asterisk}^{\perp} $ based on the decomposition~\eqref{eq:decomp}.
\end{lemma}
\begin{proof} The proof follows the same approach as the previous Lemma 3 with $\bu$ replaced by the discrete conforming solution $\bu_{\asterisk}^c$. 
 We begin by expressing $\bu_k=\bu_{k}^c +\bu_{k}^{\perp}$ according to the decomposition~\eqref{eq:decomp} and letting $\bv_k^c= \bu_*^c - \bu_k^c \in \bV_k$.
Taking advantage of the partial Galerkin orthogonality enjoyed on $\bV_k^c$, we have
\begin{align}
&\ak{(\bu_{\asterisk}^c-\bu_{k}, \bu_{\asterisk}^c-\bu_{k})} \nonumber\\
&=\ak{(\bu_{\asterisk}^c-\bu_{k},\bv_{k}^c-I_{\bV_k}^c\bv_{k}^c + I_{\bV_k}^c\bv_{k}^c-\bu_{k}^{\perp})} \nonumber \\
&=\ak{(\bu_{\asterisk}^c-\bu_{k},\bv_{k}^c-I_{\bV_k}^c\bv_{k}^c)} -\ak{(\bu_{\asterisk}^c-\bu_{k},\bu_{k}^{\perp})}. \label{local_ub1}
\end{align}
 Now due to the definition of $I_{\bV_k}^c$ described in~\eqref{interp}, we have $\bv_k^c -I_{\bV_k}^c \bv_k^c\ne0$ on $\omega(\Rka)$ while on the remaining cells of the partition $\T_k$, it is zero. Keeping this in mind and proceeding as in the proof of the reliability of the estimator with $\bu$ replaced with $\bu_*^c$ the first term of~\eqref{local_ub1} can be estimated by
 \begin{multline}
   \ak(\bu_{\asterisk}^c-\bu_{k}, \bv_{k}^c-I_{\bV_k}^c\bv_{k}^c )
   = (\bff^0 + \Delta \bu_k,\bv_k^{c} - I_{\bV_k}^c\bv_k^{c})_{\T_k}
   \\
   -({{\mvl{\partial_n \bu_k}}},\bv_k^{c} 
   - I_{\bV_k}^c\bv_k^{c})_{\F_k}
   +({{{\mathcal{L}}_S}}\bu_k, \nabla(\bv_k^{c} - I_{\bV_k}^c\bv_k^{c}))_{\T_k}
   \\
   \lesssim \big(\eta_k(\omega(\Rka))
   + \norm{h_F^{-\frac{1}{2}}\jmp{\bu_k}}_{F_k} \big)
   \norm{\nabla (\bu_*^{c} - \bu_k^{} + \bu_{k}^{\perp})}_{\T_k}.
   \label{local_ub2}
\end{multline}
For the second term of~\eqref{local_ub1}, using~\eqref{eq:best_approx}, we have 
\begin{align}
\ak{(\bu_{\asterisk}^c-\bu_{k},\bu_{k}^{\perp})} &\lesssim \gamma^{\frac{1}{2}}  \norm{h_F^{-\frac{1}{2}}\jmp{\bu_k}}_{F_k}
\norm{\bu_{\asterisk}^c-\bu_{k}}_{1,k} \nonumber\\
&\lesssim \gamma^{-\frac{1}{2}}\eta_k
\norm{\bu_{\asterisk}^c-\bu_{k}}_{1,k} \label{local_ub3}
\end{align}
for the last inequality we use the estimator control of the jump terms as described in~\eqref{eq:jumpcontrol}.
Now the result follows by combining~\eqref{local_ub2},~\eqref{local_ub3} and~\eqref{eq:best_approx}. 
\end{proof}
\subsubsection{Efficiency of the Estimator}
We begin this subsection by
recalling the definition of data oscillation for any $\bw \in \bV_k$,
\begin{gather}
\label{eq:osc}
\ok^2(\bw,\T_k):=
\norm{ h_T\big((\bff + \Delta \bw)-\Pi_2^{2m-1}(\bff + \Delta \bw)\big)}_{\T_k}^2 \nonumber\\
+ \norm{  h_F^{1/2}\big( {\mvl{ \partial_n \bw_{}}}-\Pi_2^{2m}\mvl{ \partial_n \bw_{}}\big)}^2_{\F_k},
\end{gather}
where $\Pi_2^{p}$ is the $L^2$-projection onto $\bV_k$, and $p$ denotes the highest possible degree of the polynomials characterizing $\bV_k$.
\begin{remark}[Oscillation Upper Bound]
Due to the definitions of the $L^2$-projection $\Pi_2^{p}$ and the estimator~\eqref{eq:St-estimateV0}, we have
\begin{align}
\label{osc_bound}
\ok(\bw,\T_k)\le \eta_k(\bw, \T_k).
\end{align}
\end{remark}
The following lemma presents a global lower bound for the discretization error in the DG norm error upto data oscillations. This property proves to be crucial for deriving the quasi-optimality in section 6. 

\begin{lemma}[Efficiency of the Estimator]
  Let $\bu_k$ be the IP approximation to the Stokes problem~\eqref{eq:velocity} and $\bu$ be the solution to~\eqref{eq:Stokes_V0}. Then, there exists a constant $C_{eff}>0$ depending only on the shape regularity of $\T_k$ such that
 \begin{gather}
 \label{eq:eff}
 C_{eff} \eta_k^2 \le \norm{\bu-\bu_k}_{1,k}^2+ \ok^2(\bu_k,\T_k).
 \end{gather}
\end{lemma}
\begin{proof}
To prove this lemma, we introduce interior and edge bubble functions on quadrilateral elements as described in section 3.4.1 of~\cite{AinsworthOden97} and follow the standard bubble function techniques described in~\cite{Verfuerth94} to obtain the global lower bound. 
\end{proof}
\subsection{Mark and Refine}
We begin by setting the notation of the set of marked elements by $\Mk$
i.e., $\Mk=\mathcal{M}_{\T,k} \cup \mathcal{M}_{\F,k}$ where the sets
$ \mathcal{M}_{\T,k}$ of cells $ T \in \T_k$ and $ \mathcal{M}_{\F,k}$
of faces $ F \in \F_k(\bar{\Omega})$ are marked for refinement.  As a
marking strategy for refinement we use {D\"orfler marking}, i.~e.,
given a constant $ 0 < \theta < 1$, we compute such that the following
property holds
\begin{align}
\label{eq:marking_strategy}
 \theta \ \eta_k &\le  \ \tilde{\eta}_k , \nonumber\\
\text{where } \tilde{\eta}_k := \Big( \sum\limits_{T \in \mathcal{M}_{\T,k}} \eta_T^2 \ &+  \ \sum\limits_{F \in \mathcal{M}_{\F,k}} (\eta_{F,1}^2 + \eta_{F,2}^2 ) \Big)^{1/2}.
\end{align}

Upon marking the elements for refinement, the refinement of every
quadrilateral cell $T\in\Mk$ into four children is realized by
connecting the midpoints of its edges. Of course, this refinement
strategy makes the occurrence of hanging nodes unavoidable. 
 We impose a restriction on the kind of non-geometrically, non-conformity by assuming that our refinement leads to one-irregular mesh as mentioned in section 3.
\section{Contraction property: Groundwork and proof} 
In view of our goal of establishing the convergence of the
divergence-conforming IP method, we need three main properties namely
the reliability of the estimator, an estimator reduction property and
a quasi-orthogonality for the weighted sum of the energy norm and the
estimator. Since we have already proved the reliability in the previous section, this section will be focused on proving the remaining two properties.
 
\subsection{Estimator Reduction Property}

\begin{proposition}[Estimator Reduction Property]
Let $\T_{k+1}$ and $\T_k$ be partitions of $\Omega$ such that $\T_{k+1}$ is obtained by refining { the cells in $\Mk\subset \T_k$.}
Suppose that $\bu_k \in \bV_k$ and $\bu_{k+1} \in \bV_{k+1}$ {are} the IP approximations to the Stokes Problem~\eqref{eq:velocity}.
Then, for any $\tau >0$, exists $C_{\tau}>0$ depending only on the shape regularity of $\T_k$, such that there {exists}
\begin{gather}
\label{eq:EstReduction}
\eta_{k+1}^2(\bu_{k+1}) \le \hat{\tau}\eta_k^2(\bu_k) + C_{\tau} \norm{\grad{(\bu_k - \bu_{k+1})}}^2_{{\T_k}},
\end{gather}
where $\hat{\tau}=(1+\tau)(1-\theta + \alpha)$, $\theta>0$ denoting the D\"orfler marking constant introduced in~\eqref{eq:marking_strategy}.
\end{proposition}
\begin{proof}
The proof can be done along the same lines as the proof of Corollary 3.4 in~\cite{CaKrNoSi08} and is presented here for completeness.

By using the definition of $\eta_T(\bu_{k+1})$ and $\eta_F(\bu_{k+1})$ for $T \in \T_{k+1}$
and due to the triangle inequality we have,
{
\begin{align}
\eta_{T}(\bu_{k+1})   \le& \  \eta_{T}(\bu_k) +  \ h_T \  \| \Delta (\bu_{k+1} - \bu_{k})  \|_{0,T} \label{est_1a} \\
\eta_{F}(\bu_{k+1})   \le& \eta_{F}(\bu_k) +  \norm{h_F^{1/2}\mvl{\partial_n \bu_{k+1}- \partial_n \bu_k}}_{0,F} \label{est_1b}
  \end{align}
}
 {For bounding~\eqref{est_1a}, we use the inverse estimate~\cite{BrennerScott02} for the last term on the right hand side of the inequality~\eqref{est_1a} and by squaring and applying Young's inequality with constant $\tau >0$, we have}
\begin{align}
& \eta^2_T(\bu_{k+1})   \le 
 (1 + \tau )\eta^2_T(\bu_k) + (1 + \tau^{-1} ) C_1^2\| \nabla(\bu_{k+1} - \bu_k)\|_{0,T}^2 \label{est_3} 
\end{align}
while for the inequality~\eqref{est_1b}, we square it and apply Young's inequality for the same constant $\tau>0$, sum over all edges $F \in \F_k(T)$ and finally apply the trace inequality~\cite{Warburton20032765} to obtain
\begin{align}
& \sum_{F \in \F_k(T)}\eta^2_F(\bu_{k+1})   \le 
 (1 + \tau ) \sum_{F \in \F_k(T)}\eta^2_F(\bu_k) \nonumber\\
 & +(1 + \tau^{-1} )   C_2^2 \sum\limits_{T^* \in w_T}\| \nabla(\bu_{k+1} - \bu_k)\|_{0,T^*}^2\label{est_4}
\end{align}
where $w_T$ is defined in~\eqref{defn:patch1}.
{ We note that this applies to edges with hanging nodes, since their penalty parameter is taken from the refined cell.} Now combining the two estimates~\eqref{est_3} and~\eqref{est_4}, summing over all the cells $T \in \T_{k+1}$ and employing
the finite overlap property of patches $w_T$, we obtain
\begin{align} \label{Cascon}
\eta_{k+1}^2(\bu_{k+1}) \le (1 + \tau ) \eta_{k+1}^2(\bu_k) + 2d C^2  (1 + \tau^{-1} )\| \nabla (\bu_{k+1} - \bu_k) \|_{\T_{k+1}}^2 .
\end{align}
Let $\mathcal M $ be the collection of elements in $\T_k $
which are marked for refinement.  For $T_k \in \mathcal M$, set
\begin{align*}
  \mathcal M_{k+1}(T_k) &:= \{T \in \T_{k+1} \ |  \ T \subset   T_k \}
  \\
  \text{and } \mathcal M_{k+1}
  &:=\bigcup \limits _{T_k \in \T_k} \mathcal M_{k+1}(T_k). 
\end{align*}

Since the unrefined cells live on both levels of mesh refinement, 
$\T_{k+1}  \setminus  \mathcal M_{k+1}= \T_k  \setminus  \mathcal M$. Also, as a consequence of the refinement, we have
\begin{gather}
\label{eq:refine}
\begin{split}
\sum\limits_{ T \in   \mathcal M_{k+1}(T_k)}\eta_{k+1}^2(\bu_k ,T) \  \le  \ \alpha \ \eta_k^2(\bu_k, T_k),
\end{split}
\end{gather}                                                                     
where $0 < \alpha < 1$ such that $h_{T_{k+1}} \le \alpha \ h_{T_{k}} \ \forall \ T_{k+1} \in{ \T_{k+1}\setminus \T_k, \ T_k \in \T_k}$.

As a result,   
\begin{align}
\eta_{k+1}^2(\bu_k) &=\sum\limits_{T \in \T_{k+1}\setminus \mathcal M_{k+1} } \eta_{k+1}^2(\bu_k, T) \ + \ \sum\limits_{T\in \mathcal M_{k+1} } \eta_{k+1}^2(\bu_k,T) \nonumber\\
&\le \sum\limits_{T\in \T_k} \eta_{k}^2(\bu_k, T) \  - \sum\limits_{T\in \mathcal M } \eta_{k}^2(\bu_k,T) + \ \alpha\sum\limits_{T\in \mathcal M } \eta_{k}^2(\bu_k,T), \nonumber \\
&\le ( 1- \theta^{}  +\alpha )\sum\limits_{T\in \mathcal M } \eta_{k}^2(\bu_k,T) \label{estred1}
\end{align}
where $\theta$ is the D\"orfler constant. 
Replacing $\eta_{k+1}^2(\bu_k) $ in  \eqref{Cascon} by the upper bound~\eqref{estred1} and by the definition of $\theta$, we conclude the desired inequality.
\end{proof}
We close this subsection with a result which provides an upper estimate for the oscillation term. This estimate will be used in the proofs leading to the quasi-optimality.
\begin{proposition}[Perturbation of oscillation]
Let $\T_{m}$ and $\T_k$ be partitions of $\Omega$ such that $\T_{m}$ is obtained by refining $\T_k$.
Then, we can find a constant $C_{osc}>0$ depending only on the shape regularity of $\T_{m}$ such that for any $\bv_k \in \bV_k$ and $\bv_{m} \in \bV_{m}$ we have
\begin{gather}
\label{eq:OscPerturb}
\ok^2(\bv_{k},\T_{m} \cap \T_k) \le 2\text{osc}_{m}^2(\bv_{m},\T_{m} \cap \T_k)+  C_{osc} \norm{\grad{(\bv_k - \bv_{m})}}_{\T_{m}}^2.
\end{gather}
\end{proposition}
\begin{proof}
The details of this proof are presented in~\cite{CasconKreuzerNochettoSiebert08} we skip its presentation here.
\end{proof}
\subsection{Quasi-orthogonality of the divergence-free velocity}
{   
In contrast to the adaptive $\bH^1$ conforming finite element method, the Galerkin orthogonality of the velocity does not hold true on $\bV_k$. We only have a quasi-orthogonality property relating the discretization errors corresponding to two consecutive loops.  We will derive this property in this subsection. But first, we present the following lemma which is needed for the quasi-orthogonality.
}
 \begin{lemma}[Mesh Perturbation]  Given $\bv_k \in \bV(k)$, for any $0<\varepsilon <1$, we have
 \begin{align*}
 a_{k+1}^{IP}(\bv_k,\bv_k) &\le(1+ \varepsilon)a_{k}^{IP}(\bv_k,\bv_k) \quad + \\
 &\frac{2 C_l}{\varepsilon C_L}  \big(\norm{ h_{F}^{-\frac12}\jmp{\bv_k} }_{\F_{k}}^2 +
\norm{ h_{F}^{-\frac12}\jmp{\bv_k} }_{\F_{k+1}}^2
 \big){.} 
 \end{align*}
 \end{lemma}
 \begin{proof}
 We begin by observing
 \begin{align}
  \norm{ h_{F}^{-\frac12}\jmp{\bv_k} }_{\F_{k+1}}^2&\le 2\norm{ h_{F}^{-\frac12}\jmp{\bv_k} }_{\F_{k}}^2 \label{eq_ref1} \\
  \norm{\grad{\bv_k}}_{\T_{k}}^2&=\norm{\grad{\bv_k}}_{\T_{k+1}}^2 \label{eq_ref2}
 \end{align}
 so that we have
 \begin{gather}
 \begin{split}
 a_{k+1}^{IP}(\bv_k,\bv_k)- a_{k}^{IP}(\bv_k,\bv_k) &\le -2\form(\mathcal L_S\bv_k,\grad{\bv_k})_{\T_{k+1}}+
 2 \form(\mathcal L_S\bv_k,\grad{\bv_k})_{\T_{k}} \\
 &+\gamma\norm{ h_{F}^{-\frac12}\jmp{\bv_k} }_{\F_{k}}^2
{.}\label{eq:main_l4}
 \end{split}
 \end{gather}
We now provide upper bounds for each of the terms in the right hand side of~\eqref{eq:main_l4}. For the first term, we use~\eqref{eq_ref2}, the lifting estimate~\eqref{eq:lift estimate} with constant $C_l$ and coercivity~\eqref{coer} with coercivity constant $C_L$ to obtain the upper estimate 
 \begin{align}
- 2 \form(\mathcal L_S\bv_k,\grad{\bv_k})_{\T_{k+1}} &\le 2\big(\frac{ C_l}{ C_L} 
 \norm{ h_{F}^{-\frac12}\jmp{\bv_k} }_{\F_{k+1}}^2
  \big)^{1/2}\big(C_L\norm{\grad{\bv_k}}_{\T_{k}}^2\big)^{1/2}\nonumber\\
 &\le \frac{2 C_l}{ \varepsilon C_L}\norm{ h_{F}^{-\frac12}\jmp{\bv_k} }_{\F_{k+1}}^2
  + \frac{C_L \varepsilon}{2}\norm{\grad{\bv_k}}_{\T_{k}}^2, \label{inequality_1}
 \end{align}
 where the last inequality is obtained thanks to Young's inequality with constant $\varepsilon_{}>0.$
 Similarly we derive the following bound for the second term
 \begin{gather}
 2 \form({\mathcal{L}}_S\bv_k,\grad{\bv_k})_{\T_{k}} 
 \le \frac{2 C_l}{ \epsilon C_L}\norm{ h_{F}^{-\frac12}\jmp{\bv_k} }_{\F_k}^2 + \frac{C_L \varepsilon}{2}\norm{\grad{\bv_k}}_{\T_{k}}^2, \label{inequality_2}
 \end{gather}
 with the same choice of Young's inequality constant $\varepsilon>0$.
 The result now follows by collecting the above upper bounds and by the definition of $\ak(.,.)$. {We note that again $h_F$ in case of hanging nodes is taken from the refined side.}
 \end{proof}
We are now in a position to state and prove the quasi-orthogonality result below.
 
\begin{proposition}[Quasi-orthogonality]
Let $\T_{k+1}$ and $\T_k$ be partitions of $\Omega$ such that $\T_{k+1}$ is obtained by refining $\T_k$.
Suppose that $\bu_k \in \bV_k$ and $\bu_{k+1} \in \bV_{k+1}$ {are} the IP approximations to the Stokes Problem~\eqref{eq:velocity} and $\be_{k}$ and $\be_{k+1}$ denote the discretization errors associated with these approximations respectively.
For  $ 0 < \varepsilon < {\frac{1}{2}}$ there is a constant $C_{comp} >0$ such that 
\begin{multline}
\label{eq:QuasiOrtho}
a_{k+1}^{IP}(\be_{k+1},\be_{k+1})
\le (1+\varepsilon){\ak}(\be_k,\be_k)
\\
-\frac{C_{L}}{2}\norm{\nabla{(\bu_k - \bu_{k+1})}}^2_{\T_{k+1}}
+ \frac{C_{comp}}{\gamma}\bigl(
\eta_k^2 + \eta_{k+1}^2\bigr)
\end{multline}
holds.
\end{proposition}
\begin{proof}
Since we have partial Galerkin orthogonality with respect to the space $\bV^{c}_{ k+1}$, 
we use
 $ a_{k+1}^{IP}(\be_{k+1}, \bu_{k+1}^{c} - \bu_{k}^{c}) = 0$ and express
\begin{align}
a_{k+1}^{IP}(\be_{k+1},\be_{k+1})&= a_{k+1}^{IP}(\be_{k+1} +  \bu_{k+1}^{c} - \bu_{k}^{c} ,\be_{k+1} +  \bu_{k+1}^{c} - \bu_{k}^{c}) \quad -\nonumber\\
& \quad a_{k+1}^{IP}( \bu_{k+1}^{c} - \bu_{k}^{c}, \bu_{k+1}^{c} - \bu_{k}^{c})\nonumber\\
& \le  a_{k+1}^{IP}(\be_{k+1} +  \bu_{k+1}^{c} - \bu_{k}^{c} ,\be_{k+1} +  \bu_{k+1}^{c} - \bu_{k}^{c}) \quad -\nonumber\\
& \quad C_{L}\norm{ \bu_{k+1}^{c} - \bu_{k}^{c}}_{1,k+1}^2,  \label{qo:1}
\end{align}
where we have used the coercivity of $a_{k+1}^{IP}(.,.)$ for the last
inequality.  The last term in~\eqref{qo:1} can be replaced by using
the following inequality
\begin{gather*}
  \norm{\bu_{k+1}^{c} - \bu_{k}^{c}}_{1,k+1}^2 \ge \frac12\norm{\bu_{k+1} -\bu_{k}}^2_{1,k+1}
  -\norm{\bu_{k+1}^{\perp} - \bu_{k}^{\perp}}_{1,k+1}^2.
\end{gather*}
Also, we can write  $\be_{k+1} +  \bu_{k+1}^{c} - \bu_{k}^{c} = \ { \be_{k} + \bu_{k}^{\perp} - \bu_{k+1}^{\perp}}, $ thus applying Young's inequality with $\epsilon_1>0$ we obtain
\begin{gather}
\begin{split}
a_{k+1}^{IP}(\be_{k+1},\be_{k+1}) &\le (1+\epsilon_1) a_{k+1}^{IP}(\be_{k},\be_{k}) + C_{*}(2+\epsilon_1^{-1})\norm{\bu_{k+1}^{\perp}-\bu_k^{\perp}}^2_{1,k+1} \\
& - 2^{-1}C_L\ \norm{\bu_{k+1}-\bu_k}^2_{1,k+1} \quad
\text{where} \quad C_*=\text{max}\{C_{L}, C_{U}\}\label{eq:ineq1}.
\end{split}
\end{gather}
For the second term on the right hand side of~\eqref{eq:ineq1} and~\eqref{interp} we have
\begin{align}
 \norm{\bu_{k}^{\perp} - \bu_{k+1}^{\perp}}_{1,k+1}^2 &\le 2 \norm{\bu_{k}^{\perp}}_{1,k+1}^2 + 2\norm{\bu_{k+1}^{\perp}}_{1,k+1}^2 \nonumber \\
 &\le 4 C_{\text{interp}}\norm{ h_{F}^{-\frac12}\jmp{\bv_k} }_{\F_{k}}^2 +
 2 C_{\text{interp}}\norm{ h_{F}^{-\frac12}\jmp{\bv_k} }_{\F_{k+1}}^2 \label{eq:ineq2} .
\end{align}
Finally by applying the Lemma 6 with $0< \epsilon_2 <1$, taking into account~\eqref{eq:ineq2} and using~\eqref{eq:jumpcontrol} we have
\begin{gather*}
\begin{split}
a_{k+1}^{IP}(\be_{k+1},\be_{k+1}) &\le (1+\epsilon_1)\Big( (1+ \epsilon_2)a_{k}^{IP}(\be_k,\be_k)\ + \\
& \frac{2C_l}{C_L\epsilon_2}\big(\norm{ h_{F}^{-\frac12}\jmp{\bv_k} }_{\F_{k}}^2 + 
\norm{ h_{F}^{-\frac12}\jmp{\bv_k} }_{\F_{k+1}}^2\big)\Big)\ +
\\ &C_{*}(2+\epsilon_1^{-1})\big( 4C_{\text{interp}} \norm{ h_{F}^{-\frac12}\jmp{\bv_k} }_{\F_{k}}^2 +
2C_{\text{interp}} \norm{ h_{F}^{-\frac12}\jmp{\bv_k} }_{\F_{k+1}}^2
\big)\\
& - 2^{-1}C_L\ \norm{\bu_{k+1}-\bu_k}^2_{1,k+1}
.
\end{split}
\end{gather*}
The result follows by applying~\eqref{eq:jumpcontrol}, choosing $0 < \epsilon_i < 1, \ i =1,2$ so that $\varepsilon=\epsilon_1\epsilon_2 + \epsilon_1 + \epsilon_2< \frac{1}{2}$ and setting
\begin{align}
C_{comp}=\frac{2C_J}{C_L}\max\{\frac{2C_l}{C_L\epsilon_2} , 4
C_{*}(2+\epsilon_1^{-1})C_{\text{interp}}
\}.
\nonumber
\end{align}
\end{proof}
\subsection{Contraction Property} Given the three properties of reliability of the estimator, its reduction and a quasi-orthogonality in hand, we can now present the proof of the contraction property.

\begin{proposition}[Contraction Property]
Let $\T_{k+1}$ and $\T_k$ be partitions of $\Omega$ such that $\T_{k+1}$ is obtained by refining $\T_k$.
Suppose that $\bu_k \in \bV_k$ and $\bu_{k+1} \in \bV_{k+1}$ { are} the IP approximations to the Stokes Problem~\eqref{eq:velocity} and $\be_k$ and $\be_{k+1}$ be the discretization errors in $\bV(k)$ and $\bV(k+1)$ respectively.

Then, there are constants $\rho >0$  and $0<\delta<1$ such that {for sufficiently large $\gamma$}
\begin{gather}
\begin{split}
\label{eq:contraction}
a_{k+1}^{IP}(\be_{k+1},\be_{k+1}) + \rho \  \eta_{k+1}^2 \le \delta \ (a_{k}^{IP}(\be_{k},\be_{k}) + \rho\ \eta_{k}^2)
\end{split}
\end{gather}
holds.
\end{proposition}
\begin{proof} We begin by replacing $ \norm{\grad(\bu_k - \bu_{k+1})}^2_{\T_{k+1}}$ in the quasi-orthogonality by using the estimator reduction property for $\eta_{k+1}$ to obtain:
\begin{gather*}
\begin{split}
a_{k+1}^{IP}(\be_{k+1},\be_{k+1})+ \rho
\eta_{k+1}^2&\le
 (1 +\varepsilon)a_{k}^{IP}(\be_{k},\be_{k}) +\\
 &(\frac{C_{comp}}{\gamma} + \frac{C_{L}\hat{\tau} }{2C_{\tau}} )\eta_k^2 + 
 (\rho - \frac{C_{L}}{2C_{\tau}}+ \frac{C_{comp}}{\gamma})\eta_{k+1}^2. 
\end{split}
\end{gather*}
Next, we set $\rho= \frac{C_{L}}{2C_{\tau}} - \frac{C_{comp}}{\gamma}$, choose $\gamma  > \frac{2C_{\tau}C_{comp}}{C_{L}}$ so that $\rho >0$ 
and using the reliability of  $\eta_k^2$, we have:
\begin{gather*} 
\begin{split}
a_{k+1}^{IP}(\be_{k+1},\be_{k+1})+ \rho \eta_{k+1}^2  &\le \delta \ak(\be_k,\be_k) +
\Big(
C_{rel}(1+\varepsilon - \delta) +\frac{C_{L}\hat{\tau}}{2 C_{\tau}} + \frac{C_{comp}}{\gamma}
\Big) \eta_k^2.
\end{split}
\end{gather*}
We can choose $\delta >0$ such that:
\begin{gather*}
\begin{split}
\delta \rho&=
C_{rel}(1+\varepsilon - \delta) +\frac{C_{L}\hat{\tau}}{2 C_{\tau}} + \frac{C_{comp}}{\gamma} \\
\delta &=\frac{C_{rel}(1+\varepsilon) + \frac{C_{comp}}{\gamma} +\frac{C_{L}\hat{\tau}}{2 C_{\tau}}  }
{C_{rel} +  \frac{C_{L}}{2 C_{\tau}}-\frac{C_{comp}}{\gamma} }.
\end{split}
\end{gather*}
To ensure that $\delta < 1$ we first set
\begin{gather*}
\begin{split}
\varepsilon= \frac{C_{L}(1-\hat{\tau})}{\beta C_{rel}C_{\tau}} <1 \quad  \text{with} \quad \beta>2,
\end{split}
\end{gather*}
and $\hat{\tau}<\frac{1}{4}$. Next, we find $\gamma>0$ such that
\begin{gather*}
\begin{split}
 \frac{C_{L}(1-\hat{\tau})}{\beta C_{\tau}}+ \frac{C_{comp}}{\gamma} +\frac{C_{L}\hat{\tau}}{2 C_{\tau}}  &<
  \frac{C_{L}}{2 C_{\tau}}-\frac{C_{comp}}{\gamma}
\end{split}
\end{gather*}
so that $0< \delta<1$ provided 
\begin{gather*}
\begin{split}
\gamma > \frac{4\beta C_{\tau} C_{comp}}{(1- \hat{\tau})(\beta-2)C_{L}}.
\end{split}
\end{gather*}
This completes the proof of the contraction property.
\end{proof}
{
\section{Quasi-optimality} 
The goal of this section is to present the proof of the quasi-optimal
cardinality of the adaptive IP Method in terms of the degrees of
freedom (DOFs). This will be proved by a series of lemmas which will
be described in this section.  Although the proofs of these lemmas are
a straightforward extension of the arguments presented
in~\cite{BonitoNochetto10} and~\cite{CasconKreuzerNochettoSiebert08},
we will be presenting them below for completeness.

 We begin with the quasi-optimality of the sum of the discretization error in the DG norm and the oscillation the so-called the total error for the remainder of this section. 
\begin{lemma}[Quasi-optimality of the total error]
 Let $\bu \in \bV$ and $\bu_k \in \bV_k$ solve the weak forms~\eqref{eq:Stokes_V0} and \eqref{eq:velocity} respectively. Then, for all $\bv_k \in \bV_k$ we can find a constant $C_{opt}>0$ depending only on the shape regularity of $\T_k$ such that
\begin{gather}
\label{eq:quasi_opt}
\norm{\bu-\bu_k}_{1,k}^2 + \ok^2(\bu_k,\T_k) \le C_{opt} \inf_{
\bv_k\in\bV_k}\big(\norm{\bu-\bv_k}_{1,k}^2 + \ok^2(\bv_k,\T_k)\big).
\end{gather}
\end{lemma}
\begin{proof}We prove the result by deriving upper bounds for each of the terms appearing on the left hand side of~\eqref{eq:quasi_opt}. For bounding the first term, we consider any $\bv_k\in\bV_k$ and express $\bu_k$ and $\bv_k$ as $\bu_k=\bu_k^c + \bu_k^{\perp}$ and $\bv_k=\bv_k^c + \bv_k^{\perp}$ based on the decomposition~\eqref{eq:decomp} with $\bu_k^c, \bv_k^c \in \bV_k^c$ and $\bu_k^{\perp}, \bv_k^{\perp} \in \bV_k^{\perp}$. Next, using the coercivity~\eqref{coer} of $\ak(.,.)$ we have,    
\begin{align}
\label{eq2} 
\norm{\bu - \bu_k}^2_{1,k}&\lesssim \ak(\bu - \bu_k, \bu-\bu_k)\nonumber\\
&= \ak(\bu - \bu_k, \bu - \bv_k) - \ak(\bu - \bu_k, \bu_k^0 -\bv_k^{0}) \nonumber\\
&-\ak(\bu - \bu_k, \bu_k^{\perp} -\bv_k^{\perp}).
\end{align}
Thanks to partial Galerkin orthogonality enjoyed on $\bV_k^c$, the second term on the right hand side of inequality~\eqref{eq2} vanishes. Next using the continuity~\eqref{cont} of $\ak(.,.)$ and Cauchy-Schwarz inequality, the right hand side of the inequality~\eqref{eq2} becomes
\begin{align}
\label{eq3}
\norm{\bu - \bu_k}_{1,k}&\lesssim \norm{\bu-\bv_k}_{1,k} +\norm{\bu_k^{\perp}}_{1,k} + \norm{\bv_k^{\perp}}_{1,k}.
\end{align}
The last two terms on the right hand side of~\eqref{eq3} can be estimated from above by their respective jump terms
 $\norm{ h_{F}^{-\frac12}\jmp{\bu_k} }_{\F_{k}}^2$ and 
  $\norm{ h_{F}^{-\frac12}\jmp{\bv_k} }_{\F_{k}}^2$ due to~\eqref{eq:best_approx}. Furthermore, thanks to the estimator control of the jump terms~\eqref{eq:jumpcontrol} applied to $\norm{ h_{F}^{-\frac12}\jmp{\bu_k} }_{\F_{k}}^2$ we obtain
\begin{align}
\norm{\bu - \bu_k}^2_{1,k}&\lesssim \norm{\bu-\bv_k}^2_{1,k}  + \gamma^{-1}\eta^2_k\nonumber\\
&\lesssim \norm{\bu-\bv_k}^2_{1,k}  + \gamma^{-1}\big(\norm{\bu- \bu_k}_{1,k}^2 + \ok^2(\bu_k,\T_k) \big),\label{eq4}
\end{align}
where the last inequality on the right hand side of~\eqref{eq4} follows by the efficiency of the estimator~\eqref{eq:eff}.

We now turn our attention to bounding the second term on the left hand side of~\eqref{eq:quasi_opt} which can be obtained by applying the perturbation of oscillation~\eqref{eq:OscPerturb} to $\bu_k$ and $\bv_k$ with $\T_{k+1} \equiv \T_k$ and using $\norm{\bu_k-\bv_k}_{1,k} \le \norm{\bu_k-\bu}_{1,k}+\norm{\bu-\bv_k}_{1,k} $ to obtain
\begin{align} 
\ok^2(\bu_k,\T_k) &\lesssim  \norm{\bu - \bu_k}_{1,k}^2 +
 \norm{\bu - \bv_k}^2_{1,k} +\ok^2(\bv_k,\T_k), \label{eq5}\\
 &\lesssim \ok^2(\bv_k,\T_k) + \norm{\bu-\bv_k}^2_{1,k}  + \gamma^{-1}\big(\norm{\bu- \bu_k}_{1,k}^2 + \ok^2(\bu_k,\T_k) \big), \label{eq6}
\end{align}
where we have used~\eqref{eq4} to bound the first term of~\eqref{eq5}.
Finally, gathering the upper estimates~\eqref{eq4} and~\eqref{eq6} and choosing $\gamma>0$ large enough, we obtain the inequality.
\end{proof}
In the following lemma we derive the conditions on the penalty parameter $\gamma$ and the D\"orfler marking parameter $\theta$ that guarantees the selection of optimal number of elements in the marking strategy~\eqref{eq:marking_strategy}. 

\begin{lemma}[Optimal Marking]
Let $\bu \in \bV$ solve~\eqref{eq:Stokes_V0} and let $\T_{k+1}$ and $\T_k$ be partitions of $\Omega$ such that $\T_{k+1}$ is obtained by refining $\T_k$.
Suppose that $\bu_k \in \bV_k$ and $\bu_{k+1} \in \bV_{k+1}$ be the IP approximations to the Stokes Problem~\eqref{eq:velocity} respectively.
Assume that 
\begin{gather}
\label{eq:opt_assump}
\norm{\bu - \bu_{k+1}^c}_{1,{k+1}}^2 + \text{osc}_{k+1}^2(\bu_{k+1}^c,\T_{k+1}) \le \mu \big(\norm{\bu- \bu_k}_{1,k}^2 + \ok^2(\bu_k,\T_k)\big)
\end{gather}
where $\bu_{k+1}^c=\bu_{k+1} - \bu_{k+1}^{\perp} $ is the conforming
component as described in~\eqref{eq:decomp}, and
$0 < \mu < \frac{1}{2}$ is chosen as follows
\begin{gather}
\label{eq:opt_choices}
\mu= \frac{1 - \frac{C_{ub}}{\gamma C_{eff}}\big( 2 + C_{osc} \big)}{2}\Big(1- \frac{\theta^2}{\theta_{m}^2}\Big), 
\end{gather}
where $\theta \in(0, \theta_m)$
is the D\"orfler marking parameter with
\[\theta_m=\Big(\frac{C_{eff} - C_{ub} \gamma^{-1}(2 + C_{osc})}{1+ C_{ub}(2+ C_{osc})}\Big)^{\frac{1}{2}},
\]
where $C_{eff}$, $C_{ub}$ and $C_{osc}$ are constants from efficiency~\eqref{eq:eff}, localized quasi upper bound~\eqref{eq:local_ub} and perturbation of oscillation~\eqref{eq:OscPerturb} respectively.
Then, the set $\mathbb{M}_k$ of marked elements and the marking parameter $\theta$ satisfies
\begin{gather}
\label{eq:opt_marking}
\eta_{k}^2(\bu_k,\mathbb{M}_k) \ge \theta \eta_{k}^2(\bu_k,\T_k).
\end{gather}
\end{lemma}
\begin{proof}
 Using the given assumption~\eqref{eq:opt_assump}, the efficiency of the estimator~\eqref{eq:eff}, and since  $0<\mu<\frac12,$ we have
\begin{align}
 (1-2\mu)C_{eff} \eta_k^2 &\le \norm{\bu-\bu_k}_{1,k}^2+ \ok^2(\bu_k,\T_k) - 2\text{osc}_{k+1}^2(\bu_{k+1}^c,\T_{k+1}) \nonumber\\ 
 &-2\norm{\bu- \bu_{k+1}^c}_{1,k+1}^2 \nonumber\\
 &\le 2\norm{\bu_k-\bu_{k+1}^c}^2_{1,k+1}+\ok^2(\bu_k,\T_k)- \nonumber\\
 & \quad 2\text{osc}_{k+1}^2(\bu_{k+1}^c,\T_{k+1}). \label{eq:card_marking}
\end{align}
The first term on the right hand side of the inequality~\eqref{eq:card_marking} can be bounded above using the localized quasi upper bound~\eqref{eq:local_ub} with constant $C_{ub}$. In order to bound the last two terms of~\eqref{eq:card_marking}, we first express $\ok^2(\bu_k,\T_k)=\ok^2(\bu_k,\omega(\mathbb{R}_k))+\ok^2(\bu_k,\T_k \setminus \omega(\mathbb{R}_k))$ where $\mathbb{R}_k$ denotes the set of refined elements of $\T_k$ needed to obtain $\T_{k+1}$ and $\omega(\cdot)$ is defined in~\eqref{defn:patch1} and~\eqref{defn:patch2}.
 Next, by using the oscillation upper bound~\eqref{osc_bound} and the perturbation of oscillation~\eqref{eq:OscPerturb} we arrive at the following estimate:
\begin{align}
 \ok^2(\bu_k,\T_k)&\le \eta_k^2(\bu_k, \omega(\mathbb{R}_k))+ 
  2\text{osc}_{k+1}^2(\bu_{k+1}^c,\T_{k+1}) \nonumber \\
 & + C_{osc}\norm{\bu_k - \bu_{k+1}^c}_{1,k+1}^2 \label{osc1}.
 \end{align} 
Again, applying the localized quasi upper bound~\eqref{eq:local_ub} for the last term of~\eqref{osc1} we obtain
\begin{gather}
 (1-2\mu)C_{eff} \eta_k^2(\bu_k) \le (1 + C_{ub}(2+C_{osc}))\eta_k^2(\bu_k, \omega(\mathbb{R}_k)) + \nonumber\\
  \gamma^{-1}(C_{ub}C_{osc} + 2C_{ub})\eta_k^2.
\end{gather}
The estimate now follows by plugging in the choices of $\mu$ and $\theta$ as given in~\eqref{eq:opt_choices}.
\end{proof}
Based on the total error, we can now introduce the DG and the continuous Galerkin (CG) approximation classes denoted by $\mathbb{A}_s$ and $\mathbb{A}^c_s$ respectively.  
\begin{align*}
\text {For each }N \ge 0,&\\ 
\mathcal{T}_N:=\{\T_{\asterisk} &\text{ is a partition of } \Omega \text{ generated from }\T_0 : \#\T_{\asterisk} -\#\T_0 \le N\},
\end{align*}
so that the quality of
the best approximation is measured by:
  \begin{gather*}
\sigma_N^{DG}(\bu,f):=\inf_{\T_{\asterisk} \in \mathcal{T}_N}\inf_{\bv \in \bV_{\asterisk}}\Big(\norm{\bu- \bv}_{1,\asterisk}^2 + \text{osc}_{\asterisk}^2(\bv,\T_{\asterisk})\Big)^{\frac12}.
\end{gather*} 
For $s>0$, we can define the approximation class as
\begin{gather*}
\begin{split}
&\mathbb{A}_s:=\{(\bu,\bff) \in \H^1_0(\Omega)\times L^2(\Omega): |(\bu,\bff)|_{s}<\infty\},\\
&\text{where } |(\bu,\bff)|_s:= \sup_{N>0}\big(N^s\sigma_N^{DG}(\bu,\bff )\big).
\end{split}
\end{gather*}
 As remarked in~\cite{BonitoNochetto10} and the references therein, since the approximation class denotes the set of function pairs $(\bu,\bff)$ for which the best approximation of the total error decays like $N^{-s}$, this limits the range of values of $s$ to the interval $(0, m/2]$ where $m$ denotes the order of the Raviart Thomas space. 

Analogously, we define $\sigma_N^{CG}(\bu,\bff)$ and $ |(\bu,\bff)|_s^{CG}$ as 
 \begin{gather*}
 \begin{split}
 &\sigma_N^{CG}(\bu,f):=\inf_{\T_{\asterisk} \in \mathcal{T}_N}\inf_{\bv \in \bV^c_{\asterisk}}\Big(\norm{\bu- \bv}_{1,\asterisk}^2 + \text{osc}_{\asterisk}^2(\bv,\T)\Big)^{1/2},\\
& \text{and } |(\bu,\bff)|_s^{CG}:= \sup_{N>0}\big(N^s\sigma_N^{CG}(\bu,\bff )\big),
 \end{split}
 \end{gather*}
so that the conforming approximation class is defined as
\[ 
\mathbb{A}^c_s:=\{(\bu,\bff) \in \H^1_0(\Omega)\times L^2(\Omega): |(\bu,\bff)|_{s}^{CG}<\infty\}.
\]
\begin{remark}
In order to prove the quasi-optimality for the DG approximation, we follow the standard approach of establishing the equivalence of $\mathbb{A}_s$ and $\mathbb{A}_s^c, \text{ for }0< s \le \frac{m}{2}$. The proof of this equivalence can be found in~\cite{BonitoNochetto10} we will omit its presentation here.

\end{remark}
Before stating the final lemma needed to prove the quasi-optimal rate, we make the following two assumptions on the sequence of partitions $\{\T_k\}_{k>0}$ generated from the initial mesh $\T_0$ and on $\mathbb{M}_{k}$, the set of marked elements.
\begin{assumption}[]
For a given initial mesh $\T_0$ and the sequence $\{\T_k\}_{k>0}$ generated by the recursive application of the adaptive algorithm~\eqref{algo:afem}, we can find a constant $\Lambda_0 >0$ such that 
\begin{gather}
\label{ComplexityRefine}
\#\T_k - \#\T_0 \le \Lambda_0 \sum_{j=0}^{k-1} \#\mathbb{M}_j \quad \text{holds,}
\end{gather}
where $\#\T_k$ and $\#\mathbb{M}_j$ denote the cardinality of $\T_k$ and $\mathbb{M}_j$ respectively.
\end{assumption}

\begin{assumption} 
For $ \theta \in (0,\theta_m)$, with $\theta_m$ given by~\eqref{eq:opt_assump},
\begin{center}
the set of marked elements $\mathbb{M}_{k}$ has the minimal cardinality.
\end{center}
 
\end{assumption}
Achieving the quasi-optimal complexity of the adaptive algorithm involves making the suitable choices of $\gamma$ and $\theta$ as derived in the Lemma 8 above.
 These choices coupled with Assumption 2 allow us to prove the quasi-optimal complexity of the adaptive algorithm but before that, we need the following lemma which presents an upper bound the cardinality of $\mathbb{M}_k$ in terms of the total error. 
\begin{lemma}[Cardinality of $\Mk$]Suppose that Assumption 1 and 2 hold and that 
$\mu$ and $\theta$ are chosen as in Lemma 8. Furthermore, suppose that $\bu \in \bV$ solves~\eqref{eq:Stokes} and $\bu_k \in \bV_k$ solves~\eqref{eq:lifted-ip-Stokes}. If $(\bu, \bff) \in \mathbb{A}_s$ for $0<s\le \frac{m}{2}$, then, 
\begin{gather}
\label{eq:card_M}
\#\mathcal{M}_k \lesssim \big(\norm{\bu- \bu_k}_{1,k}^2 + \ok^2(\bu_k,\T_k)\big)^{-\frac{1}{2s}}.
\end{gather}
\end{lemma}
\begin{proof} We first choose $\varepsilon>0$
 as
 \begin{gather}
 \begin{split}
 \label{eps}
 \varepsilon^2=\mu\frac{\Big(\norm{\bu- \bu_k}_{1,k}^2 + \ok^2(\bu_k,\T_k)\Big)}{C_{opt}},
 \end{split}
 \end{gather}
 where $C_{opt}$ is the constant arising from the quasi-optimality of the total error~\eqref{eq:quasi_opt}.
 
 Next, since $(\bu, \bff) \in \mathbb{A}_s\equiv\mathbb{A}_s^c$, so corresponding to the above choice of  $\varepsilon$ given by~\eqref{eps}, the definition of $\mathbb{A}_s^c$ implies that we can find $(\bu_{\varepsilon}, \T_{\varepsilon})$ with $\bu_{\varepsilon}\in \bV_{\varepsilon}\equiv\bV(\T_{\varepsilon})$ such that
 \begin{gather}
 \begin{split}
 \#\T_{\varepsilon}\lesssim \Big(\frac{|(\bu,\bff)|_s}{\varepsilon}\Big)^{\frac{1}{s}},\quad
 \norm{\bu- \bu^c_{\varepsilon}}_{1,{\varepsilon}}^2 + \text{osc}_{\T_\varepsilon}^2(\bu^c_{\varepsilon},\T_{\varepsilon}) \le \varepsilon^2.\label{card2}
 \end{split}
 \end{gather}
 In order to establish a relation between $\bu_k$ and $\bu_{\varepsilon}$ which are defined on partitions $\T_k$ and $\T_{\varepsilon}$ respectively, we introduce a partition $\T_{k\varepsilon}=\T_k \oplus \T_{\varepsilon} $ and  $\bu_{k\varepsilon}\in \bV_{k\varepsilon}\equiv\bV(\T_{k\varepsilon})$ as the corresponding solution with respect to this partition. The solution can be decomposed into $\bu_{k\varepsilon}= \bu_{k\varepsilon}^c +\bu_{k\varepsilon}^{\perp}$ according to~\eqref{eq:decomp} where the conforming component $\bu_{k\varepsilon}^c \in \bV_{k\varepsilon}^c$ satisfies
 \begin{gather*}
  a^{IP}_{{k\varepsilon}}(\bu_{k\varepsilon}^c , \bv_{k \varepsilon}^c)= (\bff,\bv_{k\varepsilon}^c) \quad \forall \bv_{k\varepsilon}^c\in \bV_{k\varepsilon}^c.
 \end{gather*}
Thus applying Lemma 7 to $\bu_{k\varepsilon}^c \in \bV_{k\varepsilon}^c$ we have
\begin{align}
 \norm{\bu-\bu_{k\varepsilon}^c}^2_{1,k\varepsilon} + \text{osc}_{\T_{k\varepsilon}}^2(\bu_{k\varepsilon}^c,\T_{k\varepsilon}) &\le C_{opt} \big(\norm{\bu-\bu^c_{\varepsilon}}_{1,\varepsilon}^2 + \text{osc}_{\T_{\varepsilon}}^2(\bu^c_{\varepsilon},\T_{ \varepsilon})\big)
 \le C_{opt} \varepsilon^2 \nonumber\\
 &=\mu\Big(\norm{\bu- \bu_k}_{1,k}^2 + {osc}_{k}^2(\bu_{k},\T_{k})\Big) \label{card1}
 \end{align}
 where the last inequality of~\eqref{card1} is obtained by using the definition of $\varepsilon$ given in~\eqref{eps} and using~\eqref{card2}.
Inequality~\eqref{card1} implies that the following D\"orfler marking property holds i.e.,
\begin{gather*}
\begin{split}
\eta_{k}^2(\bu_k,\mathbb{R}_{k{\varepsilon}}) \ge \theta \eta_{k}^2(\bu_k,\T_k), \quad 0 < \theta< \theta_m.
\end{split}
\end{gather*}
Since the set of marked elements $\mathbb{M}_k$ has the minimum cardinality,
we have
\begin{gather}
\begin{split}
\label{card3}
\#\mathbb{M}_k &\lesssim \#\mathbb{R}_{{kk \varepsilon}}\le \#\T_{k \varepsilon}-\#\T_{k},
\end{split}
\end{gather}
where $\mathbb{R}_{{kk \varepsilon}}$ denotes the set of refined cells needed to obtain $\T_{k\varepsilon}$ from $\T_k$.
Also, due to the definition of $\T_{k \varepsilon}, $ we have
\begin{gather}
\begin{split}
\#\T_{k \varepsilon}-\#\T_k &= \#\T_{ \varepsilon}- \#\T_0 \le \T_{ \varepsilon}\\
&\lesssim \Big(\frac{|(\bu,\bff)|_s}{\varepsilon}\Big)^{\frac{1}{s}} \label{card4}.
\end{split}
\end{gather}
where~\eqref{card4} holds due to~\eqref{card2}. Finally, combining the estimates~\eqref{card3} and~\eqref{card4} and using the definition of $\varepsilon$ given in~\eqref{eps}, we obtain the result.
\end{proof}
%
With all the prerequisites in hand, we are in a position to state and proof the main result of this section. The proof follows the same approach as presented in~\cite{BonitoNochetto10} and we present it below for completeness.
\begin{theorem}[Quasi-Optimality]
Suppose that the marking parameter $\theta \in (0, \theta_m)$ where $\theta_m$ is defined in Lemma 8  and suppose that Assumptions 1 and 2 hold true.
For a given $\bff \in L^2(\Omega)$, suppose that $\bu \in \bH_0^1(\Omega)$ solves the weak form~\eqref{eq:Stokes_V0} of the Stokes problem. Furthermore, let the triple 
$\{(\T_k, \bV_k,\bu_k )\}_{k \ge 0}$ denote the sequence of partitions, discrete solution spaces and solutions respectively generated by a recursive application of adaptive algorithm~\eqref{algo:afem}. Furthermore, suppose that $(\bu,\bff) \in \mathbb{A}_s$. Then, for small enough mesh size $h_0$ corresponding to the initial mesh $\T_0$, there exists $C>0$ satisfying
\begin{gather}
\big(\norm{\bu -\bu_k}^2_{1,k} + \ok^2(\bu_k, \T_k)\big)^{1/2} \le C \#\T_k^{-s}.
\end{gather}
\end{theorem}
\begin{proof}
We begin with by observing that $\#\T_k \lesssim  
\#\T_k - \#\T_0 $ thus using assumption 1 in conjunction with Lemma 9 we have
\begin{align}
\#\T_k &\lesssim \sum_{j=0}^{k-1} \#\mathbb{M}_j 
\lesssim \sum_{j=0}^{k-1} \big(\norm{\bu- \bu_j}_{1,j}^2 + \text{osc}_j^2(\bu_j,\T_j)\big)^{-\frac{1}{2s}}\nonumber\\
&\lesssim\sum_{j=0}^{k-1} \big(\norm{\bu- \bu_j}_{\aj}^2 + \eta_{j}^2(\bu_j,\T_j)\big)^{-\frac{1}{2s}} \label{eq:qo1}
\end{align}
where the last inequality~\eqref{eq:qo1} above holds due to the efficiency of the estimator and the continuity of $\aj(.,.)$. Finally, applying the contraction property~\eqref{eq:contraction} we arrive at 
\begin{align}
\#\T_k  &\lesssim \sum_{j=0}^{k-1}\big(\delta^{1/s}\big)^j \big(\norm{\bu- \bu_k}_{\ak}^2 + \eta_{k}^2(\bu_k,\T_k)\big)^{-\frac{1}{2s}}\nonumber\\
&\lesssim  \frac{1}{\delta^{1/s}}\big(\norm{\bu- \bu_k}_{\ak}^2 + \eta_{j}^2(\bu_k,\T_k)\big)^{-\frac{1}{2s}}
, \quad \text{since } \delta <1,\\
&\lesssim \frac{1}{\delta^{1/s}}\big(\norm{\bu -\bu_k}^2_{1,k} + \ok^2(\bu_k, \T_k)\big)^{-\frac{1}{2s}} \label{eq:qo2}.
\end{align}
To obtain the last inequality~\eqref{eq:qo2}, we have used the oscillation upper bound~\eqref{osc_bound} and the coercivity~\eqref{coer}.
\end{proof}
}
\section{Numerical Results}
In this section we present the results of the well known academic example which was proposed in~\cite{Verfuerth96} on the notorious L-shaped domain
\begin{gather*}
\Omega= (-1,1)^2 \setminus(0,1)^2
\end{gather*}
with the divergence-free forcing function and Dirichlet boundary conditions chosen so that the exact velocity and pressure in the polar coordinates $(r,{\Theta})$ are
\begin{gather*}
\begin{split}
\bu(r,{\Theta})&:= r^{\lambda} \left( \begin{array}{c}
( (1+ \lambda)\sin({\Theta})\Psi({\Theta}) + \cos({\Theta})\Psi^{\prime}({\Theta}) \\
\sin({\Theta})\Psi^\prime({\Theta})  - (1 + \lambda)\cos({\Theta})\Psi({\Theta})
\end{array}  
\right),\\
p &:= -r^{\lambda -1}\frac{(1+ \lambda)^2\Psi^{\prime}({\Theta}) + \Psi^{\prime \prime \prime}({\Theta})}{(1 - \lambda)},\\
\end{split}
\end{gather*}
 where,
 \begin{gather*}
 \begin{split}
 \Psi({\Theta}) \ &= \ \sin((1 + \lambda){\Theta})\cos(\lambda \omega)/(1 +\lambda ) \ - \cos((1+ \lambda){\Theta}) \\
  		& -  \sin((1 - \lambda){\Theta} )\cos(\lambda \omega)/(1 - \lambda ) \ + \ \cos((1- \lambda){\Theta} ), \\
\omega \ &= \frac{3 \pi}{2}  \qquad \text {and, }\lambda \ \approx \ 0.5448373678246.
 \end{split}
 \end{gather*}

Although $\grad \bu$ and $p$ admit a singularity at the re-entrant corner, we are concerned with the singularity for the velocity $\bu$ which is of the form $r^\lambda$.

To demonstrate the numerical performance of the adaptive method, we report the error decays in {Figure}~\ref{fig:conv hist} and the refinement history of the velocity dependent estimator in~\ref{fig:fig1} employing biquadratic, bicubic and biquartic Raviart Thomas elements and for different values of the constant $\theta$ in the D\"orfler marking.

\begin{figure}[p]
\centering
\includegraphics[width=0.45\textwidth,height=0.45\textwidth]{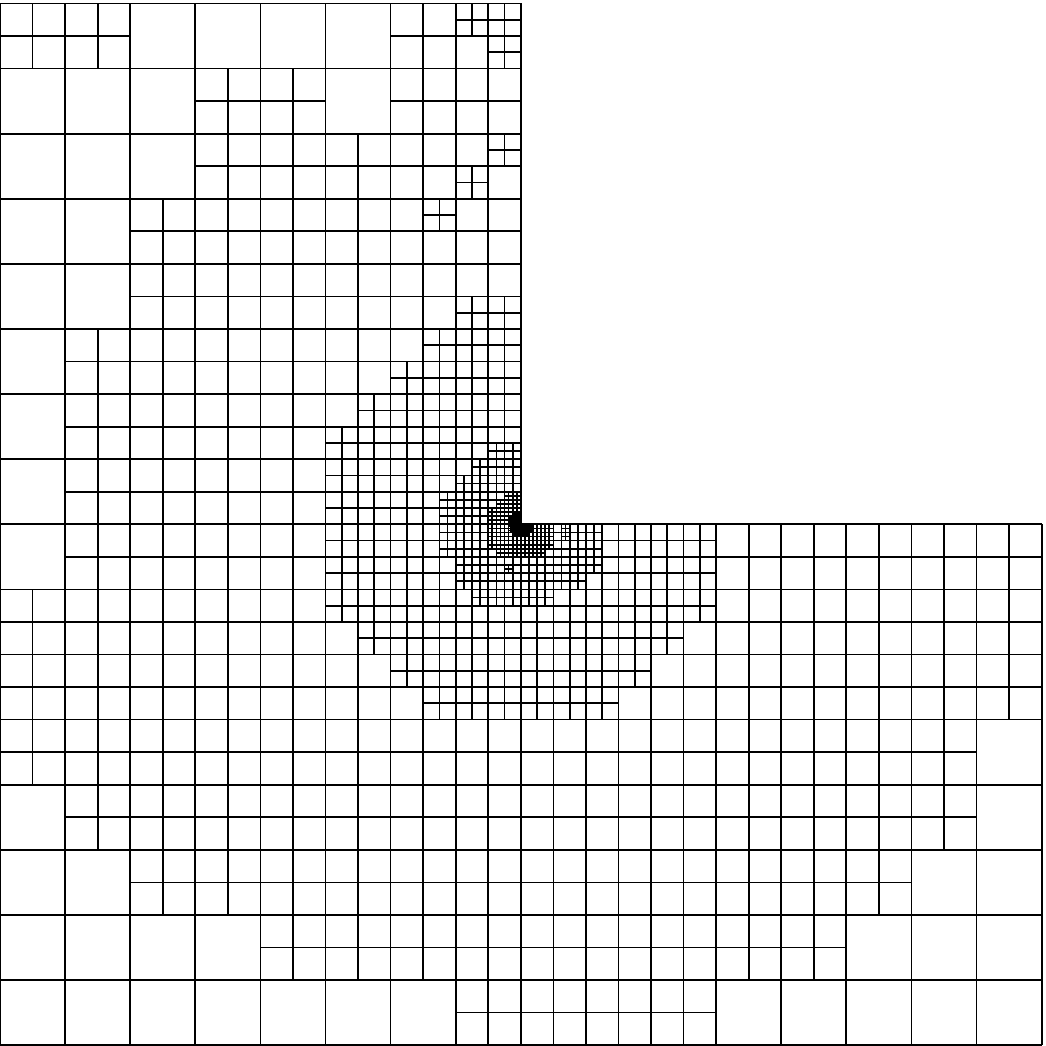} 
\includegraphics[width=0.45\textwidth,height=0.45\textwidth]{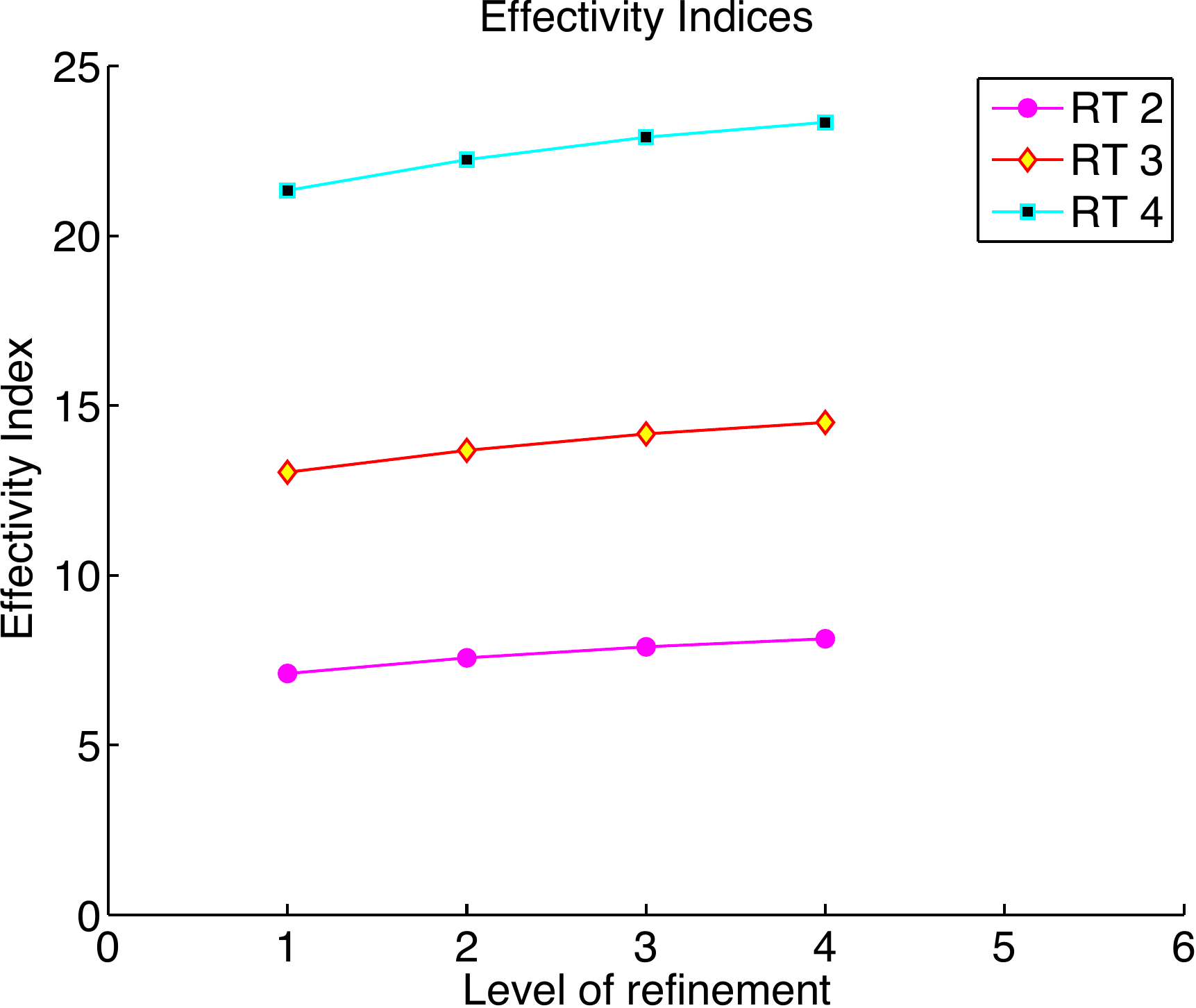} 
\caption{ Adaptive mesh using fourth order Raviart Thomas element after eight levels of adaptive refinement for $\theta =0.5$ (left) and effectivity indices for the different orders of Raviart Thomas approximation (right). }
\label{fig:fig1}

\end{figure}

\begin{figure}[p]
  \centering
  \includegraphics[width=.496599\textwidth]{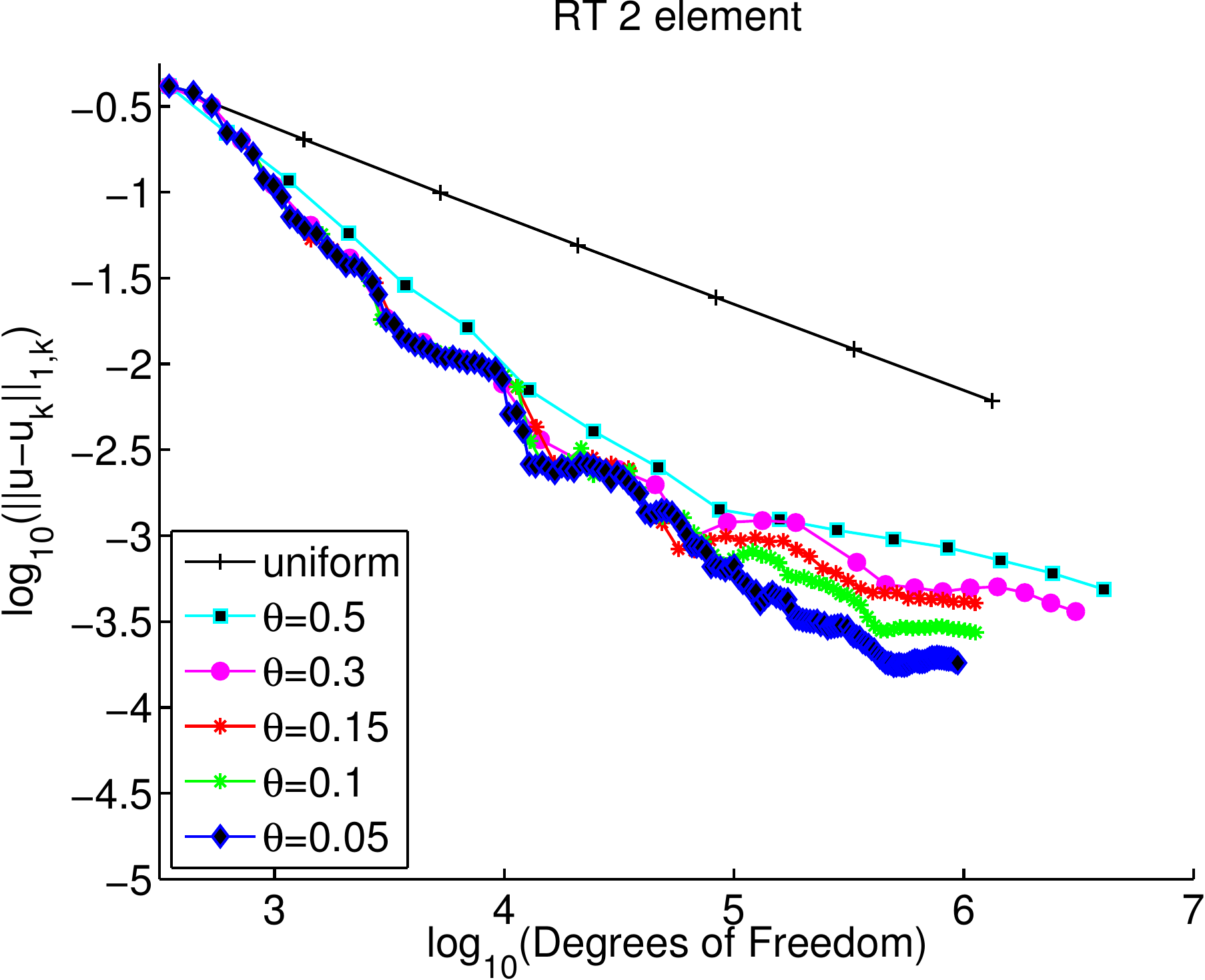}
   \includegraphics[width=.496599\textwidth]{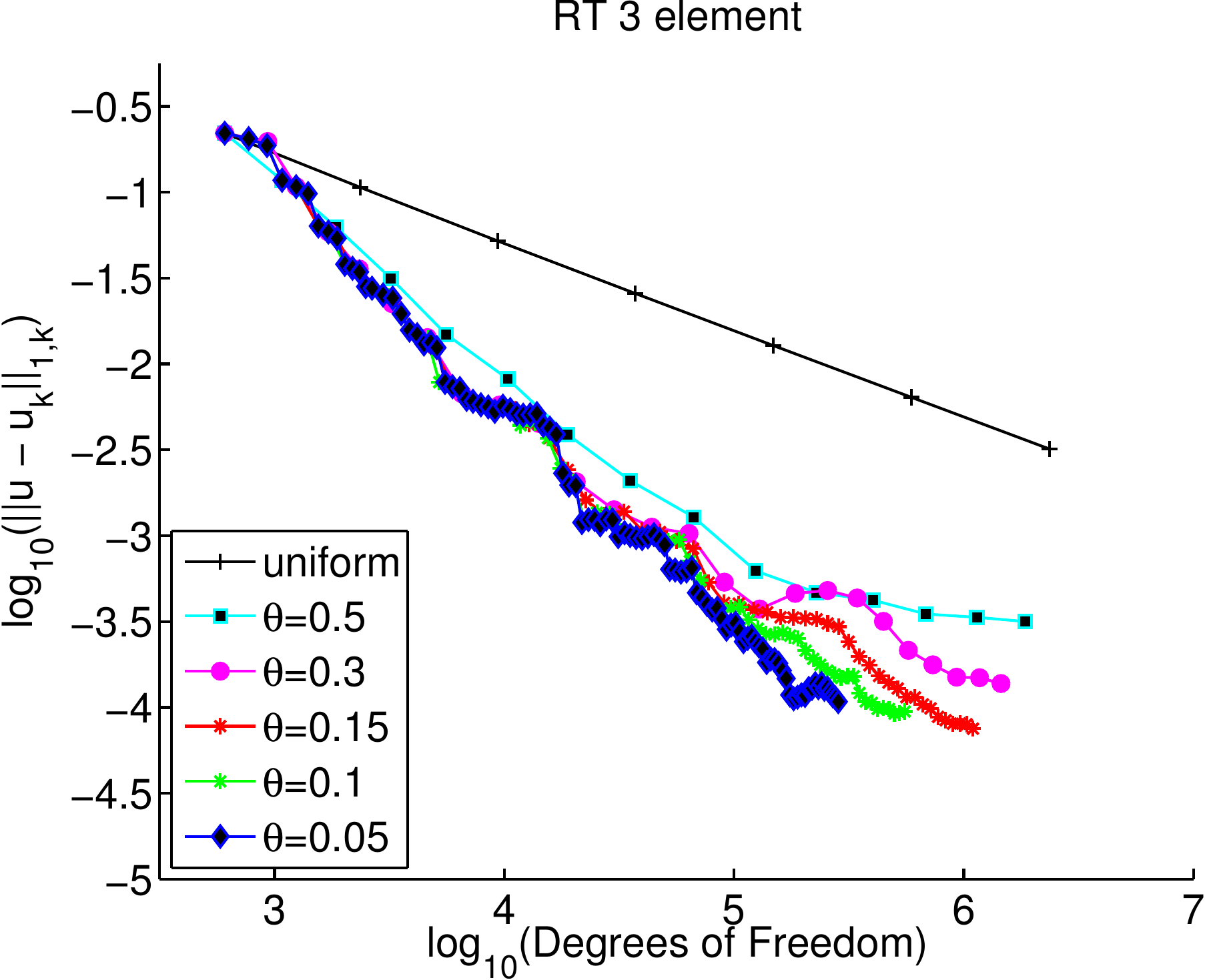}
     \includegraphics[width=.496599\textwidth]{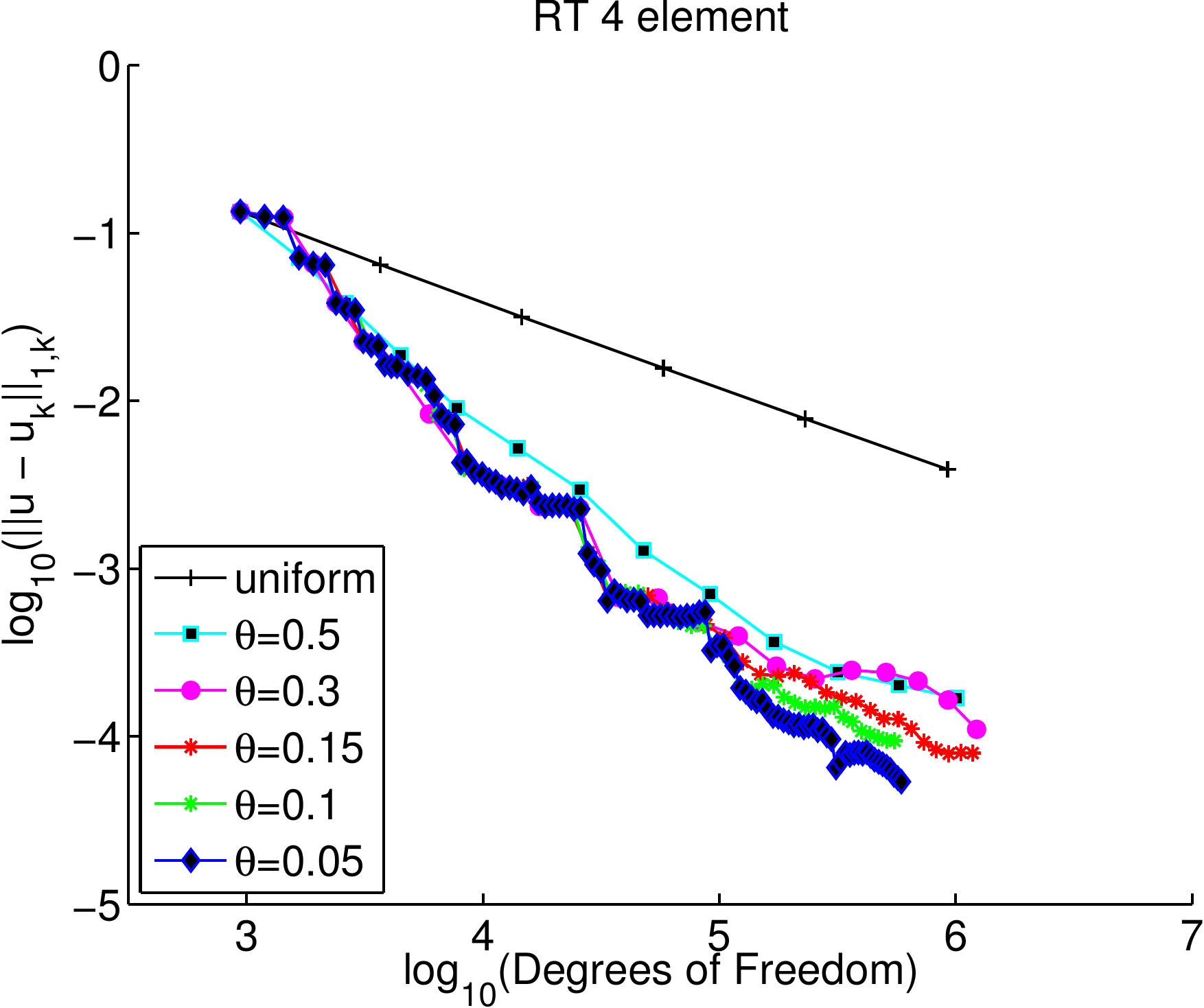}
  \caption{The convergence histories with rates for biquadratic (top left), bicubic (top right) and biquartic (bottom center) Raviart Thomas elements.}
  \label{fig:conv hist}
\end{figure}
{
Figure~\ref{fig:conv hist} reflects the convergence history of the adaptive algorithm with the discrete error displayed as a function of the DOFs on a logarithmic scale and for different choices of $\theta$.} The curves
represent the decrease in the error both for uniform refinement and for adaptive refinement. We note that for lower polynomial degree we are unable to retrieve the optimal convergence rate  $\text{(DOFs)}^{m/2}$, {\it{m}} being the order of the Raviart Thomas element. However, by employing a higher order of Raviart Thomas element, we can retrieve the optimal convergence rate.

In conclusion, we presented a contraction property for the Stokes problem by relying on the discrete Hodge decompositions of velocity. {Furthermore, the quasi-optimal cardinality for the Stokes problem was also presented.}

\section*{Acknowledgements}
Computations were performed using the deal.II finite element software library~\cite{BangerthHartmannKanschat07,dealII84}

\bibliography{all,kanschat,local}{}
\bibliographystyle{plain}
\end{document}